\documentclass[11pt,a4paper]{article}
\usepackage{amsmath}
\usepackage{amsfonts}
\usepackage{amsthm}
\usepackage{amssymb}
\usepackage{url}
\usepackage[usenames,dvips]{color}
\usepackage{comment}
\usepackage[dvips]{graphicx}
\usepackage{psfrag, float}  

\newtheorem{lemma}{Lemma}[section]
\newtheorem{corollary}[lemma]{Corollary}

\newtheorem{remark}[lemma]{Remark}

\newtheorem{notation}[lemma]{Notation}
\newtheorem{proposition}[lemma]{Proposition}
\newtheorem{theorem}[lemma]{Theorem}

\def\be{\begin{eqnarray}}
\def\ee{\end{eqnarray}}
\def\beal{\begin{aligned}}
\def\enal{\end{aligned}}
\newcommand{\eps}{\varepsilon}
\newcommand{\ga}{\gamma}

\newcommand{\RR}{\mathbb{R}}
\newcommand{\NN}{\mathbb{N}}
\newcommand{\CC}{\mathbb{C}}
\newcommand{\TT}{\mathbb{T}}

\newcommand{\ZZ}{\mathbb{Z}}

\newcommand{\PP}{\mathcal{P}}
\newcommand{\GG}{\mathcal{G}}

\newcommand{\BB}{\mathcal{B}}
\newcommand{\LL}{\mathcal{L}}

\newcommand{\SSS}{\mathcal{S}}

\newcommand{\XX}{\mathcal{X}}
\newcommand{\OO}{\mathcal{O}}
\newcommand{\FF}{\mathcal{F}}

\newcommand{\VV}{\mathcal{V}}

\newcommand{\UU}{\mathcal{U}}
\newcommand{\NNN}{\mathcal{N}}

\newcommand{\AAA}{\mathcal{A}}

\newcommand{\CCC}{\mathcal{C}}

\newcommand{\Id}{\mathrm{Id}}

\newcommand{\ii}{^{-1}}
\newcommand{\de}{\delta}
\newcommand{\pa}{\partial}

\newcommand{\al}{\alpha}

\newcommand{\rr}{\rho}

\newcommand{\ol}{\overline}

\newcommand{\La}{\Lambda}

\renewcommand{\Re}{\mathrm{Re\, }}

\newcommand{\wt}{\widetilde}
\newcommand{\wh}{\widehat}

\newcommand{\h}{\mathrm{h}}

\begin{document}

\title{Oscillatory orbits in the restricted
elliptic planar three body problem}

\author{Marcel Guardia\thanks{\tt marcel.guardia@upc.edu}, Pau Mart\'\i
n\thanks{\tt martin@ma4.upc.edu}, Lara Sabbagh\thanks{\tt
l.el-Sabbagh@warwick.ac.uk}\ and
Tere M. Seara\thanks{\tt tere.m-seara@upc.edu}}
\maketitle

\medskip
\begin{center}$^{* \S}$
Departament de Matem\`atica Aplicada I\\
Universitat Polit\`ecnica de Catalunya\\
Diagonal 647, 08028 Barcelona, Spain
\end{center}
\smallskip
\begin{center}$^\dagger$
Departament de Matem\`atica Aplicada IV\\
Universitat Polit\`ecnica de Catalunya\\
Campus Nord, Edifici C3, C. Jordi Girona, 1-3. 08034 Barcelona, Spain
\end{center}
\smallskip
\begin{center}$^\ddagger$
Mathematics Institute\\
Zeeman Building, University of Warwick\\
Coventry CV4 7AL
\end{center}

\begin{abstract}
The restricted planar elliptic three body problem  models the motion
of a massless body under the Newtonian gravitational force of  two
other bodies, the primaries, which evolve in Keplerian ellipses.

A trajectory  is called  oscillatory if it leaves every bounded
region but  returns
infinitely often to some fixed bounded region. We prove the existence of such
type of trajectories for any values for the masses of the primaries provided
the eccentricity of the Keplerian ellipses is small.


\end{abstract}

\tableofcontents

\section{Introduction}\label{sec:intro}
The restricted planar elliptic three body problem (RPE3BP from now on)
  models the motion
of a body of zero mass under the Newtonian gravitational force of  two
other bodies, the primaries, which evolve in
Keplerian ellipses with eccentricity $e_0$. Without loss of generality one can
assume that the masses of the primaries are  $\mu$ and $1-\mu$,
their period is $2\pi$ and their positions are $-\mu q_0(t)$
and
$(1-\mu)q_0(t)$, where
\[
 q_0(t)=(\rr(t)\cos v(t), \rr(t)\sin v(t))
\]
with
\[
 \rr(t)=\frac{1-e_0^2}{1+e_0\cos v(t)}
\]
and $v(t)$ is called the \emph{true anomaly}, which satisfies $v(0)=0$ and
\[
 \frac{dv}{dt}=\frac{(1+e_0\cos v)^2}{(1-e_0^2)^{3/2}}.
\]
Then, the motion of the third body is described by the following Hamiltonian
\begin{equation}\label{def:Ham:Cartesian}
  H(q,p,t;e_0)= \frac{\|p\|^2}{2} -  V(q,t;e_0)
\end{equation}
where
\[
 V(q,t;e_0)= \frac{1-\mu}{\left\|q+\mu
      q_0(t)\right\|}+\frac{\mu}{\left\|q-(1-\mu)q_0(t)\right\|}
\]
where $q,p\in\RR^2$.

The eccentricity of the ellipses satisfies $e_0\in [0,1)$. If $e_0=0$, the
primaries describe circular orbits. This case is known as  the restricted planar
circular three body
problem (RPC3BP from now on). In this paper we consider $e_0>0$ and small. Note
that the Hamiltonian
also depends on $\mu\in[0,1/2]$. We do not write this dependence explicitly
since for us $\mu$ is a fixed positive parameter.

The purpose of this paper is to analyze some particular orbits of this
Hamiltonian system: the \emph{oscillatory motions},
that is, orbits which leave every bounded region but  return
infinitely often to some fixed bounded region. To prove the existence of such
orbits, we use the  framework usually considered in the study of Arnold
diffusion in nearly integrable Hamiltonian systems.




If one takes $\mu=0$, the system reduces to a central force
problem which is integrable and therefore cannot have oscillatory motions.
That is, they can take place
for the three body problem but not in its limit  $\mu\rightarrow
0$. In this paper we prove that oscillatory motions are possible for any
value of $\mu\in (0,1/2]$ and $e_0>0$ small enough.  The existence of such
motions for any $\mu$ and  $e_0=0$ (the circular problem) has been
recently proved by the authors of this paper in \cite{GuardiaMS14}.

To obtain oscillatory orbits we work in a  nearly integrable setting and we use
perturbative methods that provide
invariant objects which can be analyzed and act as  a skeleton
that such orbits follow. Nevertheless, since we give
results
for any $\mu\in (0,1/2]$, the nearly integrable setting cannot be in
terms
of  $\mu\rightarrow0$. Instead, we consider the following regime. Take
the body of
zero mass very far away from the two primaries. Then, at first order the
third body
perceives the two primaries as just one body at the center of mass plus a
periodic perturbation whose smallness comes from the ratio of the distance
between
the two primaries over the distance between the third body and the center of
mass. Then, taking the ratio of distances small enough,  one has an
integrable Hamiltonian plus a small periodic perturbation  (see
\cite{GuardiaMS14}).

\subsection{Final motions in the three body problem}
One of the most important questions in the analysis of the three body
problem (either restricted or non restricted, planar or spatial) is the study
of the
final motions. That is, what type of behaviors can happen as time $t\rightarrow
\pm \infty$. Its analysis was initiated by Chazy in 1922, when he gave a
complete classification of the possible final motions (see Section 2.4 of
\cite{ArnoldKN88}). In the restricted setting, the possible final motions are
the following:
\begin{itemize}
 \item $H^\pm$ (hyperbolic):  $\|q(t)\|\rightarrow\infty$ and $\|\dot
q(t)\|\rightarrow c>0$ as $t\rightarrow\pm\infty$.
\item $P^\pm$ (parabolic): $\|q(t)\|\rightarrow\infty$ and $\|\dot
q(t)\|\rightarrow 0$ as $t\rightarrow\pm\infty$.
\item $B^\pm$ (bounded): $\limsup_{t\rightarrow \pm\infty}\|q\|<+\infty$.
\item $OS^\pm$ (oscillatory): $\limsup_{t\rightarrow \pm\infty}\|q\|=+\infty$
and  $\liminf_{t\rightarrow \pm\infty}\|q\|<+\infty$.
\end{itemize}
Examples of all these behaviors, except the oscillatory motions,
were already known by Chazy. The other three behaviors certainly already exist
for the two body problem, where motion is confined to conics in the state space
and hyperbolic motions arises on hyperbolas, parabolic motions on parabolas and
bounded motion on ellipses. As already explained, oscillatory motions cannot
appear in the two body problem.

The study of oscillatory motions was initiated by Sitnikov in the sixties
\cite{Sitnikov60}. He proved their existence in a very carefully
chosen symmetrical model consisting of two bodies of equal mass
revolving in planar ellipses around their center of mass and a third
body of mass zero moving along the perpendicular axis at the center
of mass.  His work also proved   that $X^-\cap
Y^+\neq\emptyset$ for $X,Y=H,P,B,OS$.

Later, Moser~\cite{Moser01} gave a new proof of the Sitnikov result.
The subsequent  results in the area strongly rely on the ideas developed by
Moser. The present paper certainly also relies on some of them but
also uses other ideas recently developed for the study of Arnold
diffusion. Using
Moser ideas, Llibre and Sim\'{o}~\cite{LlibreS80} obtained
oscillatory motions for the collinear three body problem.

For the planar three body problem, the first result was also by Llibre and
Sim\'o \cite{SimoL80}, where they proved the existence of oscillatory motions
for the RPC3BP for
small enough values of the mass ratio $\mu$ (Hamiltonian
\eqref{def:Ham:Cartesian} with $e_0=0$). Using a real-analyticity argument, Xia
\cite{Xia92} extended their result to any $\mu \in (0, 1/2]$ except a finite
(unknown) number of values (Moser
\cite{Moser01} had previously used this argument in the Sitnikov problem).
 The first general result was obtained by three of the
authors of the present paper \cite{GuardiaMS14}, who proved the existence
of
oscillatory motions of the restricted planar circular three body problem for
any value of the mass ratio $\mu \in (0, 1/2]$.

A completely different approach using Aubry-Mather theory and semi-infinite
regions of instability was developed in \cite{GalanteK11,
GalanteK10b, GalanteK10c}. The authors considered the RPC3BP with a realistic
mass ratio for the Sun-Jupiter system. Using
computer-assisted methods, they proved the existence of orbits with
initial
conditions in the range of our Solar System which become oscillatory as time
tends to infinity.

All the mentioned results deal with systems which can be reduced to a two
dimensional area preserving map. This makes the proof of existence of
oscillatory motions considerably simpler than if one considers other three body
problems (restricted elliptic, restricted spatial, non-restricted) which have
more degrees of freedom. For higher dimensional systems, Moser ideas are
 harder to apply and
the results are more scarce.

In more degrees of freedom the first result is due to
Alexeev~\cite{Alekseev68} (actually before Moser ideas) who
generalized Sitnikov result~\cite{Sitnikov60} to small mass for the
third body. Positive mass increases the number of degrees of freedom
of the system to three. Moeckel~\cite{Moeckel84, Moeckel07} has proven results
concerning oscillatory motions in the regime of small angular momentum (and
thus close to triple collision).

Concerning the RPE3BP
\eqref{def:Ham:Cartesian}, Robinson gave a conditional result on the
existence of oscillatory orbits~\cite{Robinson84, Robinson15}. He proved their
existence provided certain homoclinic points were present. As far as
the authors know, the existence of such points has not been proved
yet.

A fundamental problem concerning oscillatory motions is to measure
how abundant they are. As pointed out in~\cite{GorodetskiK12}, in
the conference in honor of the 70th anniversary of Alexeev, Arnold
posed the following question: \emph{Is the Lebesgue measure of the
set of oscillatory motions positive?} Arnold considered this
problem the central problem of celestial mechanics. Alexeev
conjectured in \cite{Alekseyev71} that the Lebesgue measure is zero
(in the English version~\cite{Alekseev81} he attributes this
conjecture to Kolmogorov). This conjecture remains wide open.

The only result dealing with  the abundance of oscillatory motions is
the recent paper~\cite{GorodetskiK12}. The authors study
the Hausdorff dimension of the set of oscillatory motions for the
Sitnikov example and the RPC3BP. Using~\cite{Moser01}
and~\cite{SimoL80}, they prove that  the Hausdorff dimension of the set of
oscillatory motins is maximal for a Baire
generic subset of an open set of parameters (the eccentricity of the
primaries in the Sitnikov example and the mass ratio and the Jacobi
constant in the RPC3BP).

\subsection{Arnold diffusion in the three body problem and growth in angular
momentum}

As we have said, we study a regime where the RPE3BP is close to a two body
problem. In the latter,  the angular momentum is a
first integral. This fact is no
longer true in the former. Thus, a natural question is whether
the angular momentum of the zero mass body only varies by a small amount (with
respect to the
perturbative parameter) or can make big excursions.

This question fits into the framework of Arnold diffusion \cite{Arnold64}.
Consider a nearly integrable system in action-angle coordinates
\[
 H(\varphi,I)=H_0(I)+\eps H_1(\varphi,I),\quad \varphi\in\TT^n,\, I\in
V\subset\RR^n,\,\eps\ll 1.
\]
For $\eps=0$, the action variables $I$ are constants of motion. Arnold diffusion
analyzes the drastic changes that the actions can undergo for small $\eps>0$.

In the setting of nearly integrable  $N$-body problems, the existence of
Arnold diffusion can be analyzed in very different regions in the phase space
and for very different ranges of the involved parameters. As far as the authors
know, the first paper dealing with Arnold diffusion in Celestial Mechanics is
\cite{Moeckel96} (later completed in \cite{Zheng:2010}) who considers
the five body problem. In \cite{DelshamsGR11}, the authors analyze unstable
behavior for the three body problem close to the Lagrangian point $L_1$.

Concerning to the growth of angular momentum, the paper \cite{FejozGKR11}
obtains such behavior for the RPE3BP (Hamiltonian
\eqref{def:Ham:Cartesian}) along the mean motion resonances, which implies a
change of eccentricty in the osculating ellipse of the body of mass zero. More
related to our paper, is the recent \cite{DelshamsKRS14}, where such behavior is
obtained in a neighborhood of parabolic motions, the so-called invariant
manifolds of infinity. In \cite{DelshamsKRS14}, it is proven the existence of
orbits whose angular momentum $G(t)$ satisfies $G(0)<G_1$, $G(T)>G_2$ for some
$T>0$ and for any given $G_2>G_1$. Nevertheless, they need to assume that $\mu$
is exponentially small with respect to $G_2$ and $e_0$ is polynomially small
with respect to $G_2$.


We expect that this result can be generalized to any value of the mass ratio.
Nevertheless, in the present paper we only give a conditional result (see
Remark \ref{rmk:RangeParametersDiffusion}).


\subsection{Main result}

We give in this section the main result that we obtain.

\begin{theorem}\label{thm:Main}
Fix any $\mu\in (0,1/2]$. There exists  $e^*_0(\mu)>0$ such that for any $e_0\in
(0,e_0^*(\mu))$ there exists an orbit $(q(t), p(t))$
of~\eqref{def:Ham:Cartesian} which is
oscillatory. Namely, it satisfies
\[
 \limsup_{t\rightarrow +\infty}\|q\|=+\infty\,\,\,\text{ and }\,\,\,
\liminf_{t\rightarrow +\infty}\|q\|<+\infty.
\]
\end{theorem}

This theorem shows that $OS^+\neq\emptyset$. Proceeding analogously, one can
show that $OS^-\neq\emptyset$. Nevertheless, our techniques do not allow us to
show that $OS^-\cap OS^+\neq\emptyset$. To do that, one needs to consider more
sophisticated techniques. Indeed, the ideas developed by Moser rely on the
construction of a horseshoe which has branches arbitrarily close to infinity.
This allows to combine all possible past and future behavior. Here, as
we explain in Section \ref{sec:Comparison}, we rely on different
techniques, which are simpler to generalize to higher dimensions:
shadowing certain invariant objects. These techniques can only be applied to one
time direction, either future or past, but not to both of them at the same
time.

We only obtain oscillatory motions provided the primaries perform nearly
circular orbits. It is expected that such motions exist for any value of the
eccentricity $e_0\in [0,1)$. The approach presented in this paper can be
applied to this more general setting. The only additional difficulty is to
extend the  proof of  transversality of the invariant manifolds of infinity to
this wider range of parameters (see Section \ref{sec:InvManifoldsInfinity} for
more details). In this paper, we use perturbative arguments and use that this
transversality is already known for the circular problem \cite{GuardiaMS14}.


\begin{remark}\label{rmk:RangeParametersDiffusion}
The geometric framework we use to prove Theorem~\ref{thm:Main} can
be also used to obtain orbits with large drift in the angular
momentum, provided some non-degeneracy condition is satisfied. If such
condition could be verified to be true, one would generalize the results in
\cite{DelshamsKRS14} by obtaining the following. Fix $\mu\in (0,1/2]$. There
exists $G_0(\mu)>0$ such that for any
$G_2\geq G_1\geq G_0(\mu)$, there exists  $e^*_0(G_2)>0$ such that for any
$e_0\in(0,e_0^*(G_2))$ there  exist a time $T>0$ and  an orbit $(q(t), p(t))$
of~\eqref{def:Ham:Cartesian} whose angular momentum $G$
satisfies
\[
G(0)\leq G_1\quad\text{and}\quad G(T)\geq G_2.
\]
The needed non-degeneracy
condition is explained in Remark
\ref{rmk:ConditionDiffusion}.
\end{remark}

\subsection{Common framework for oscillatory motions and Arnold
diffusion}\label{sec:Comparison}
The purpose of this section is to relate the results of both Theorem
\ref{thm:Main} and Remark \ref{rmk:RangeParametersDiffusion} and put them in the
same framework.
This is explained in more detail in Section \ref{sec:InvManifoldsInfinity}.

Let us start by explaining Moser ideas to obtain oscillatory orbits for the
Sitnikov example
\cite{Moser01} and applied in \cite{SimoL80} to the RPC3BP. The RPC3BP
(Hamiltonian \eqref{def:Ham:Cartesian} with $e_0=0$) has a first integral, the
Jacobi constant. Then, in
suitable coordinates, the RPC3BP can be reduced to a two dimensional Poincar\'e
map for which ``infinity'' $\{|q|=+\infty, \dot q=0\}$ is a parabolic fixed
point.

Assume for a moment that this fixed point is hyperbolic. Then, it would have
stable and unstable invariant manifolds. Assume that these invariant manifolds
intersect transversally. Then, Smale Theorem would ensure that there exists a
horseshoe. Suitable orbits in this horseshoe  travel close to the
invariant manifolds and the lim inf of the distances to the fixed point is
zero. Such orbits, in the original coordinates, are oscillatory.

The infinity point of the RPC3BP is not hyperbolic but parabolic.
\cite{SimoL80}, as done in
\cite{Moser01} for the Sitnikov problem, shows that even if it is parabolic one
can carry out
the same strategy. Note that one has to face different difficult issues: prove
the existence of the invariant manifolds of the parabolic point (see
\cite{McGehee73}), prove that
they intersect transversally and prove a Lambda lemma for parabolic points
which implies the existence of symbolic dynamics.

To carry out this strategy in the elliptic case is certainly more involved
since the phase space has dimension five and therefore one cannot reduce the
dynamics to an area preserving map. The stroboscopic Poincar\'{e} map is
four dimensional.
Thus,  infinity cannot be reduced to a fixed
point but it forms  a cylinder with one angular variable and one real variable
(see Section \ref{sec:InvManifoldsInfinity}). This cylinder is ``normally
parabolic'' and it has invariant manifolds. Since in this paper we consider the
elliptic problem as a perturbation of the circular one, the result in
\cite{GuardiaMS14} implies that the invariant manifolds of this cylinder
intersect transversally.

To prove that such transversal intersections lead to oscillatory motions we do
not rely on the construction of symbolic dynamics as Moser did \cite{Moser01}.
Instead, we use the standard method of Arnold diffusion to construct a
transition chain of tori \cite{Arnold64} (in the present setting of fixed
points).
That is, we find a sequence of fixed points belonging to
the cylinder of infinity which are connected by transversal heteroclinic
orbits.  In Arnold diffusion problems usually it is enough to construct a finite
chain.
Instead, to obtain oscillatory orbits one has to construct an infinite chain
and then prove the existence of an orbit which shadows the chain.
This construction is much simpler than a horseshoe and leads to oscillatory
motions forward (or backward) in time. Certainly a horseshoe gives much more
information and imply a plethora of different types of motion. In particular,
it
allows to combine different types of final motions in the past and in the
future
(including motions which are oscillatory as both time tends to plus and
minus infinity). Our simpler techniques can only show the existence of
oscillatory motions in one  direction of time. Nevertheless, we think that one
of the main
interests of this paper is to show that such a simple mechanism as the
existence of transtion chains leads to oscillatory motions.

The vertical direction of the cylinder of infinity can be parameterized by the
angular momentum of the orbits. To prove the existence of oscillatory orbits
one needs to construct a chain which remains in a compact portion of this
cylinder. The reason is that orbits shadowing chains with unbounded angular momentum cannot be
oscillatory since $\liminf |q|=+\infty$. In fact, we
can prove the existence of chains confined in very thin portions of the
cylinder and thus with almost constant angular momentum. 
This is done by considering the scattering map \cite{DelshamsLS08} and
studing its dynamical properties.
Now, one can ask the opposite question. It is possible to construct a chain which
implies
a large deviation in the angular momentum? Orbits shadowing such chains have a
big drift in angular momentum and present the phenomenon of Arnold
diffusion as explained in Remark \ref{rmk:RangeParametersDiffusion}.
Summarizing, the
orbits
given  both by  Theorem \ref{thm:Main} and Remark
\ref{rmk:RangeParametersDiffusion} are
obtained thanks to suitable transition chains associated to the cylinder of
infinity.

The structure of the paper goes as follows. First in Section
\ref{sec:InvManifoldsInfinity} we analyze the invariant manifolds of infinity.
We prove its existence and regularity and we analyze their transversal
intersections. This allows us to construct the needed transition chain of
periodic orbits. In Section \ref{sec:Shadowing}, we state a Lambda lemma which
can be applied to the invariant manifolds of the normally parabolic cylinder of
infinity. This Lambda lemma allows us to prove Theorem \ref{thm:Main}. The
prove of the Lambda Lemma is deferred to Section
\ref{sec:proof_of_lem:LambdaLemma}.

\section{The invariant manifolds of infinity}\label{sec:InvManifoldsInfinity}
Let  $(r,\al)$ be the polar
coordinates in the plane and  $(y, G)$ their symplectic conjugate momenta.
Then, the Hamiltonian \eqref{def:Ham:Cartesian} becomes
\[
\wt H(r,\al,y,G,t;e_0) = \frac{1}{2} \left( \frac{G^2}{r^2}+ y^2\right)
-  U(r,\al,t;e_0),
\]
where $U(r,\al,t;e_0)=V(re^{i\al},t;e_0)$.
The variable $G$ is the angular momentum of the third body.
As we deal with a non-autonomous system, we add the
equation $\dot s=1$ to the Hamiltonian equations of Hamiltonian $\widetilde H
(r,\al,
y,G,s;e_0)$.

Since we want to study the invariant manifolds of infinity, we consider the
McGehee coordinates $(x,\al,y,G)$
where
\[
 r=\frac{2}{x^2}, \text{
for }x>0.
\]
We obtain the new system
\begin{equation}\label{eq:3bpInMcGehee}
\begin{aligned}
\dot{x}&=-\frac{1}{4}x^{3}y
&\dot{y}&=\frac{1}{8}G^{2}x^{6}-\frac{x^{3}}{4}\frac{\partial
\mathcal{U}}{\partial x}\\
\dot{\alpha}&=\phantom{-}\frac{1}{4} x^{4}G    &\dot{G}&=\frac{\partial
\mathcal{U}}{\partial \alpha}\\
\dot s & =1
\end{aligned}
\end{equation}
where the potential $\mathcal{U}$ is given by
\[
\mathcal{U}(x,\alpha,s;e_0)=U(2/x^2,\alpha,s;e_0)=\frac{x^2}{2}\left(\frac{1-\mu
}{
\sigma_S}+\frac{\mu}{\sigma_J}\right)
\]
with
\begin{align*}
 |q-q_\textrm{S}|^{2}=\sigma_\textrm{S}^{2}&=1-\mu r_0(s) x^{2}\cos(\alpha -
v(s))+\frac{1}{4} \mu^{2} r_0^{2}(s)x^{4},\\
 |q-q_\textrm{J}|^{2}=\sigma_\textrm{J}^{2}&=1+(1-\mu) r_0(s) x^{2}\cos(\alpha -
v(s))+\frac{1}{4}(1-\mu)^{2} r_0^{2}(s)x^{4}.
\end{align*}
We write the potential as
\begin{equation}\label{def:PerturbedPotential}
 \UU=\frac{x^2}{2}+\Delta\UU.
\end{equation}
Since
\[
 U(r,\al,t;e_0)=\frac{1}{r}-\frac{\mu(1-\mu)}{2}(1-3\cos
(\al-v(t))\frac{\rr^2(t)}{r^3}+\OO\left(\frac{\mu}{r^4}\right),
\]
the potential $\Delta\UU$ satisfies $\Delta\UU=\OO(\mu x^{6})$.

In view of \eqref{eq:3bpInMcGehee}, now ``infinity''
is foliated by the parabolic periodic orbits
\[
\widetilde
\Lambda_{\al_0,G_0} = \{(x,\alpha,y,G,s) \in \RR\times \TT\times
\RR\times \RR\times \TT\mid x = y = 0, \; \al = \al_0, \; G = G_0\}.
\]


Next theorem claims that each periodic orbit $\widetilde \Lambda_{\al_0,G_0}$
has stable and
unstable $2$-dimensional invariant manifolds
$W^s(\widetilde \Lambda_{\al_0,G_0})$ and $W^u(\widetilde
\Lambda_{\al_0,G_0})$.

\begin{theorem}\label{thm:ExistenciaVarietats}
Let $\widetilde \phi_t$ the flow of the
system~\eqref{eq:3bpInMcGehee} and define the projections
$\pi_{x}(x,\alpha,y,G,s) = x$ and $\pi_{(x,y)}(x,\alpha,y,G,s)
= (x,y)$. Let $(\alpha_0,G_0) \in \TT\times
\RR$. There exists $\rho_0$ such that for any $0<\rho <\rho_0$, the
local stable set
\begin{multline*}
W^s_{\rho}(\widetilde \Lambda_{\al_0,G_0}) = \{ (x,\alpha,y,G,s) \in
\RR\times \TT\times \RR\times \RR\times \TT\mid \pi_{x}
\phi_t(x,\alpha,y,G,s)
>0, \\ \; |\pi_{(x,y)} \widetilde \phi_t(x,\alpha,y,G,s)|\le \rho,\; \lim_{t\to
+\infty} \text{\rm dist}(\widetilde
\phi_t(x,\alpha,y,G,s),\widetilde \Lambda_{\al_0,G_0})= 0\},
\end{multline*}
is a $2$-dimensional manifold.

Moreover, there exists $u_0>0$ such
that $W^s_{\rho}(\widetilde \Lambda_{\al_0,G_0})$ admits a
$\CCC^{\infty}$ parametrization $\gamma^s_{\al_0,G_0}:[0,u_0)\times \TT
\to \RR\times \TT\times \RR\times \RR\times \TT$, with
$\gamma_{\al_0,G_0}^s(0,s) = (0,\alpha_0,0,G_0,s)$, analytic in
$(0,u_0)\times \TT$, which depends analytically on $(\alpha_0,G_0)
\in \TT\times \RR$.
The analogous result for the (local) unstable set also holds.
\end{theorem}

The proof of this theorem follows from
Proposition \ref{prop:local_invariant_manifolds} in Section
\ref{sec:proof_of_lem:LambdaLemma} (the technical details are done in Appendix
\ref{sec:invman}).  A related result concerning the invariant manifolds of
infinity is given in \cite{Robinson84} (see also \cite{Robinson15}).
Nevertheless we have included the proof
of this theorem, which is based in the parameterization method
\cite{CabreFL03a,BaldomaFdlLM07}, for two reasons. One the one hand, we are not
able to follow all the details in the proof in \cite{Robinson84,Robinson15}. On
the other
hand, our result provides the regularity of the stable/unstable foliations of
the invariant manifolds, needed in Lemma \ref{lem:LambdaLemma}.

Theorem \ref{thm:ExistenciaVarietats} shows that the points which tend
asymptotically in forward (or
backward) time to the periodic orbit $\widetilde \Lambda_{\al_0,G_0}$ form  a
manifold.
Since this periodic orbit is not hyperbolic  but
parabolic the rate of convergence  is not exponential and its invariant
manifolds are
not analytic at the periodic orbit but only $\CCC^\infty$. They are
analytic at any other point.


As we have explained in Section \ref{sec:intro}, when $\mu=0$, the
invariant manifolds  $W^s(\widetilde \Lambda_{\al_0,G_0})$ and
$W^u(\widetilde \Lambda_{\al_0,G_0})$ coincide and form a
two-parameter family of parabolas in the configuration space (in the
original cartesian coordinates).
We study the splitting of these invariant manifolds for any $\mu\in(0,1/2]$ in
two steps. First for the
circular problem ($e_0=0$) and then for the elliptic problem ($0<e_0\ll1$). In this
study  we need formulas for the homoclinic manifolds when $\mu=0$. For the
derivation of such formulas, one can see \cite{SimoL80}.

\begin{lemma}\label{lemma:UnperturbedSeparatrix}
Take $\mu=0$ and fix $\al_0\in\TT$ and $G_0\neq 0$. Then, system
\eqref{eq:3bpInMcGehee} has a family of homoclinic orbits to the
periodic orbit $\widetilde \Lambda_{\al_0,G_0}$, which is
given by
\[
 \begin{split}
 x_\h(t;G_0)=&\frac{2}{G_0(1+\tau^2)^{1/2}}\\
 y_\h(t;G_0)=&\frac{2\tau}{G_0(1+\tau^2)} \\
 \al_\h(t;\al_0)=&\al_0+ \wt \al_\h(t),\quad \wt
\al_\h(t)=2\arctan \tau\\
 G_\h(t;G_0)=&G_0 \\
 s_\h(t;s_0)=&s_0+t \\
 \end{split}
\]
where  $s_0\in\TT$ is a free parameter and $\tau$ and the time
$t$ are related through
\[
 t=\frac{G_0^3}{2}\left(\tau+\frac{\tau^3}{3}\right).
\]
\end{lemma}
From now on, we abuse notation and we consider the functions $x_\h$,
$y_\h$ and $\wt\al_\h$ both as functions of $\tau$ or $t$.


\subsection{Scattering map for the circular problem in the extended phase space}
\label{sec:scatering_circular}

Using the results of Theorem \ref{thm:ExistenciaVarietats}, for any $G_1>0$, the invariant set
\begin{equation}\label{def:cylinderLambda}
\widetilde \Lambda^{[G_1,+\infty)} = \bigcup_{\al_0\in \TT, G_0\ge
G_1}\widetilde
\Lambda_{\al_0,G_0} = \{(x,\alpha,y,G,s) \in \RR\times \TT\times
\RR\times \RR\times \TT\mid x = y = 0, \, G\ge G_1\}
\end{equation}
is a ``normally parabolic'' $3$-dimensional invariant manifold with
stable and unstable $4$-dimensional invariant manifolds
\[
W^\varsigma(\widetilde \Lambda^{[G_1,+\infty)})=\bigcup_{\al_0\in\TT, G_0\ge
G_1}W^\varsigma(\widetilde \Lambda_{\al_0,G_0}),\,\,\,\varsigma=u,s.
\]
Theorem 2.2 of \cite{GuardiaMS14} implies that, when $e_0=0$ and $\mu\in (0,1/2]$,  there exists $G^*\gg
1$ such that, for any $G_1
>G^*$, the invariant manifolds $W^s(\widetilde \Lambda^{[G_1,+\infty)})$ and $W^u(\widetilde \Lambda^{[G_1,+\infty)})$ intersect
transversally  in the whole space along two different $3$-dimensional homoclinic manifolds.

Following \cite{DelshamsLS08}, this transversality allows us to
define two  scattering maps $\widetilde\SSS_0^{\pm}$ associated to the
two
different transversal homoclinic intersections between
$W^s(\widetilde \La^{[G_1,+\infty)})$ and $W^u(\widetilde \La^{[G_1,+\infty)})$
and to obtain
formulas for these maps. To this end, we consider the Poincar\'e
function for $e_0=0$,
\begin{equation}
\label{def:poincarepotential}
 L(\alpha_0,G_0,s_0,\sigma;0)=
\int_{-\infty}^{\infty}\Delta\mathcal{U}(x_\h(\sigma+t;G_0),
\alpha_0+\wt\alpha_\h(\sigma+t;G_0),s_0+t;0)\, dt
\end{equation}
where $\Delta\UU$ is the potential defined in \eqref{def:PerturbedPotential} and
$(x_\h, \wt\al_\h)$
are  components of the parameterization of the unperturbed separatrix
given in Lemma \ref{lemma:UnperturbedSeparatrix}.

For any $(\alpha_0,G_0,s_0)$, the function $\sigma\mapsto
L(\alpha_0,G_0,s_0,\sigma;0)$ has two critical points $\sigma^*_\pm$
given by:
\begin{equation}
\label{def:sigmastarpm}
 \sigma^*_-=s_0-\al_0,\,\, \sigma^*_+=\pi+s_0-\al_0.
\end{equation}
This fact is given by Proposition 3.1 of  \cite{GuardiaMS14} (note that this
proposition is stated in certain scaled variables).

The reason to obtain such simple formulas for the critical points is
twofold. On the one hand, when $e_0=0$,  the potential only depends
on the angles $\alpha_0$ and $s_0$ through $\al_0-s_0$ . Thus, since
\[
 L(\alpha_0,G_0,s_0,\sigma;0) =  L(\alpha_0,G_0,s_0-\sigma,0;0),
\]
the Poincar\'e function only depends on one angular variable
$\al_0-s_0+\sigma$. On the other hand, its Fourier expansion only
contains cosines (see the Appendix of \cite{GuardiaMS14}).

Associated to the zero $\sigma^*_+$, there is a heteroclinic
connection between two periodic orbits  in $\widetilde
\Lambda^{[G_1,+\infty)}$, which are
$\mu G_0^{-4}$--close to $\widetilde
\Lambda_{\alpha_0,G_0}$ (analogously for  $\sigma^*_-$). These
heteroclinic connections satisfy
\[
\widetilde
\Gamma_\pm^{\al_0,G_0,s_0}(t)=\left(x_\h(\sigma_\pm^*+t;G_0),
y_\h(\sigma_\pm^*+t;G_0),\alpha_0+\alpha_\h(\sigma_\pm^*+t;G_0),G_0,
s_0+t\right)+\OO\left(\mu G_0^{-4}\right)
\]
(see~\cite{GuardiaMS14}). Note that these heteroclinic orbits are well
defined for any $\al_0\in\TT$, $s_0\in\TT$ and $G_0\geq G_1$ (see
\eqref{def:cylinderLambda}). This
implies that there is a homoclinic channel (see \cite{DelshamsLS06a}) which is defined for all
points in $\widetilde \Lambda^{[G_1,+\infty)}$. Thus, we can define global
scattering maps
\[
\widetilde  \SSS_0^\pm:\widetilde
\Lambda^{[G_1,+\infty)}\longrightarrow\widetilde
\Lambda^{[G_1,+\infty)}.
\]
Following  \cite{DelshamsLS08}, recall that $x_+=\widetilde  \SSS_0^+(x_-)$ if there exists a heteroclinic connection between these two points
through the prescribed homoclinic channel (see \cite{DelshamsLS08} for a more precise definition and properties of the
scattering map). Usually it is not possible to define globally the scattering
map since it is only defined locally in open sets. Then, globally, it can be
multivaluated (see \cite{DelshamsLS06a}). The particular form of the
circular problem allows us to define it globally. Note also that, in principle,
the scattering map should map $\widetilde
\Lambda^{[G_1,\infty)}$ to a bigger cylinder. Nevertheless, as it is shown in
the next proposition, in the circular case the image is the same cylinder.


\begin{proposition}\label{prop:ScatteringMapCircular}
Let $G_1>G^*$. The scattering maps $\widetilde \SSS_0^\pm: \TT\times
 [G_1,+\infty)\times \TT \longrightarrow \TT\times
 [G_1,+\infty)\times \TT$  are of the
following form,
\[
\widetilde  \SSS_0^\pm (\al, G, s)= ( \al+f^\pm (G), G,s),
\]
 where
\begin{equation}\label{def:twist}
f^\pm (G)
=- \mu(1-\mu)\frac{3\pi}{2G^4}+\OO\left(G^{-8}\right).
\end{equation}
\end{proposition}
This proposition is proven in Appendix \ref{app:scattering}.

\subsection{Reduction to the Poincar\'e map}

We reduce the dimension of the system by considering the
stroboscopic-Poincar\'e map associated to the section $\Sigma = \{s
= s_0\}$,
\begin{equation}\label{def:Poincare_map}
\begin{array}{cccc}
 \PP:& \Sigma & \longrightarrow & \Sigma\\
      & (x,y,\al,G) & \mapsto & \PP (x,y,\al,G)
 \end{array}
\end{equation}
Then $\Lambda_{\al_0,G_0} =\widetilde \Lambda_{\al_0,G_0} \cap
\Sigma$ is a two parameter family of parabolic fixed points of $\PP$
with $1$-dimensional stable and unstable manifolds
\[
W^\varsigma(\Lambda_{\al_0,G_0}) = W^\varsigma(\widetilde
\Lambda_{\al_0,G_0}) \cap \Sigma, \quad \varsigma = u,s.
\]
Analogously, $\Lambda^{[G_1,+\infty)} = \widetilde \Lambda^{[G_1,+\infty)} \cap
\Sigma$ is the
$2$-dimensional normally parabolic invariant cylinder of infinity
with $3$-dimensional invariant stable and unstable manifolds
\[
W^\varsigma(\Lambda^{[G_1,+\infty)}) = W^\varsigma(\widetilde
\Lambda^{[G_1,+\infty)}) \cap \Sigma,
\quad \varsigma = u,s,
\]
which, for $e_0=0$, intersect transversally along two $2$-dimensional homoclinic
channels. The two scattering maps associated to these homoclinic
channels are given by
\begin{equation}\label{def:scatteringcircular}
 \SSS_0^\pm (\al, G)= ( \al+f^\pm (G), G),
\end{equation}
where $f^\pm$ is the function given in~\eqref{def:twist}. They do not
depend on the section~$\Sigma$.

\subsection{Scattering map of the elliptic problem}

Once we have analyzed the splitting of the invariant manifolds  of
$\Lambda^{[G_1,+\infty)}$ for  the circular problem and derived formulas for the two
scattering maps, now we consider the elliptic problem for $e_0$
small enough. Note that $e_0$ is a regular parameter of the elliptic
problem and, therefore, we can apply classical perturbative
arguments to the stroboscopic-Poincar\'e map~$\PP$ (see~\cite{DelshamsLS08}).

Let $G_2>G_1>G^*$ (see \eqref{def:cylinderLambda} and Proposition
\ref{prop:ScatteringMapCircular}) be fixed. We call $\Lambda^{[G_1,G_2]}$ to
$\Lambda^{[G_1,+\infty)} \cap \{G_1\le G \le G_2\}$, which is compact and
invariant.  
Then, for
$e_0$ small enough, the stable and unstable manifolds of $\Lambda^{[G_1,G_2]}$
intersect transversally. Note that the smallness of $e_0$ depends on the chosen interval. The perturbative arguments imply that
there are two global homoclinic channels as in the circular problem.
These two channels define two scattering maps
\[
 \SSS^\pm:
\Lambda^{[G_1,G_2]}\longrightarrow \Lambda^{[G^*,+\infty)},
\]
which depend regularly
on $e_0$,
\begin{equation}\label{def:ScatteringElliptic}
 \SSS^\pm=\SSS^\pm_0+e_0\SSS^\pm_1+\OO\left(e_0^2\right),
\end{equation}
where $\SSS^\pm_0$ are the scattering maps of the circular problem
given by~\eqref{def:scatteringcircular}.  The theory
developed in~\cite{DelshamsLS08} does not directly apply to this
case since the invariant cylinder is not hyperbolic, but parabolic. Nevertheless,
the arguments in~\cite[Proposition~4]{DelshamsKRS14} show that the
theory of scattering maps do apply also to this problem. Hence, the
maps $\SSS^\pm$ are area preserving maps on the cylinder.

To construct oscillatory orbits of the elliptic problem, we need an
infinite transition chain: a sequence  of fixed points belonging to $\La^{[G_1,G_2]}$
 with  transversal heteroclinic connections
between consecutive points of the chain. By the definition of the scattering map, any bounded (forward or backward)  orbit of one of the scattering maps
provides such a chain.

To obtain bounded orbits of the scattering maps we observe that, if
$e_0$ is small enough, the maps $\SSS^\pm$ possess invariant curves.
Indeed, in view of~\eqref{def:twist}, $\SSS^\pm$ are twist
maps if $e_0$ is small enough. Then, we can apply  the following twist theorem,
due to Herman, from~\cite{Herman83c}:

\begin{theorem}[Twist theorem]\label{thm:KAM} Let $f:[0,1]\times\TT\rightarrow
    [0,1]\times\TT$ be an exact symplectic $\CCC^l$ map with $l>4$.
    Assume that $f=f_0+\de f_1$, where $f_0(I,\psi)=(I,\psi+A(I))$,
    $A$ is $\CCC^l$, $|\pa_I A|>M$ and $\|f_1\|_{\CCC^l}\leq 1$.
    Then, if $\de^{1/2} M^{-1}=\rho$ is sufficiently small, for a set
    of $\omega$ of Diophantine numbers of exponent $\theta=5/4$, we
    can find invariant curves which are the graph of $\CCC^{l-3}$
    functions $u_\omega$, the motion on them is $\CCC^{l-3}$ conjugate
    to the rotation by $\omega$, and
    $\|u_\omega\|_{\CCC^{l-3}}\leq C\de^{1/2}$.
    \end{theorem}

In our setting we have that the twist condition satisfies $|\pa_G A|\gtrsim
\mu G^{-5}$ whereas the $\de\lesssim e_0$. Therefore, for any
$G_2>G_1>G^*$ and $e_0$ small enough, the map
$\SSS^\pm$ has KAM
curves inside $\Lambda^{[G_1, G_2]}$. Any orbit
of~$\SSS^\pm$ in any of these KAM curves is bounded, and so are
all orbits between any two of these KAM curves.

\begin{corollary}\label{coro:boundedorbitscattering}
Consider any constants $G_2>G_1>G^*$.
Then, for $e_0$ small enough (which may depend on $G_2$ and $G_1$),
the scattering maps $\SSS^\pm$ considered
in~\eqref{def:ScatteringElliptic} have orbits which remain for all
time in $\Lambda^{[G_1,G_2]}$.
\end{corollary}

As we have explained, coming back to the original Poincar\'e map
\eqref{def:PoincareMap}, this bounded orbit of the scattering map
corresponds to a sequence of fixed points
$\{\Lambda_{\alpha_k,G_k}\}_{k\in\NN}\subset \Lambda^{[G_1,G_2]}$ such
that
$W^u(\Lambda_{\alpha_k,G_k})$ intersects transversally
$W^s(\Lambda^{[G_1,G_2]})$ at a point $P$ which belongs to
$W^s(\Lambda_{\alpha_{k+1},G_{k+1}})$.

\begin{remark}\label{rmk:ConditionDiffusion}
To obtain oscillatory motions we have looked for an infinite
transition chain with bounded angular momentum, or equivalently,  for a bounded
orbit of one of
the scattering maps $\SSS^\pm$. If we want orbits with a drift
in the angular momentum (see Remark
\ref{rmk:RangeParametersDiffusion}), we have to look for a
transition chain connecting  fixed points of $\PP$ in
\eqref{def:PoincareMap} with a large difference in angular momentum.
In this case it suffices to  look for a finite transition chain.

For the oscillatory motions,  we have used the fact
that the scattering maps  $\SSS^\pm$ are nearly integrable and
possess KAM tori. KAM tori act as barriers for the orbits of the scattering
maps and therefore,  using only one scattering map,
it is impossible to construct orbits with a large drift in $G$.
The usual strategy to
prove  the  existence of
Arnold diffusion (see for instance \cite{Arnold64, DelshamsLS06a}) is to
combine one scattering map with the inner dynamics of the invariant cylinder
induced
by the
Poincar\'e map \eqref{def:PoincareMap}. Combining these two dynamics one
obtains Arnold diffusion. Here this approach is not possible since the inner
dynamics is trivial (the cylinder $\Lambda^{[G_1,G_2]}$ is filled by fixed
points).
Thus, we rely on an idea developed in \cite{DelshamsKRS14} which is to combine
the two scattering maps $\SSS^+$ and $\SSS^-$.

In \cite{Calvez07}, Le Calvez showed the following: consider
two area preserving twist maps on a cylinder. Assume that they do not have
common invariant curves. Then,  combining them one can obtain orbits with a
drift in the action component. Thus, as long as the maps $\SSS^\pm$ did
not have common invariant curves, we could obtain transition chains with a
large drift in $G$. This fact has proven to be true for $\mu$ small enough in
\cite{DelshamsKRS14}. Nevertheless, it is not straightforward
to verify it  in the present setting. That is the non-degeneracy
assumption mentioned in Remark \ref{rmk:RangeParametersDiffusion}.
\end{remark}

\section{Shadowing orbits}\label{sec:Shadowing}

The second ingredient in the proof of Theorem~\ref{thm:Main} is what
is usually called a Lambda or Shadowing lemma. That is, to analyze
how orbits close to the stable manifold of one of the fixed points
of~$\La^{[G_1,G_2]}$ (see Theorem~\ref{thm:ExistenciaVarietats})
evolve and stretch along the unstable invariant manifold of the
fixed point. Such results allow to shadow a concatenation of
heteroclinic orbits which connect different (or the same) fixed
points.

Usually in the literature of Arnol'd diffusion, one considers more
general Lambda lemmas which allow to shadow invariant manifolds of
more general objects  (see for instance \cite{Marco96, Cresson97,
FontichM98,Cresson01,Gidea04a,Gidea04b,Bolotin06, DelshamsGR13, Sabbaghh13}).
Nevertheless, such results deal with invariant manifolds of objects
which are (normally or partially) hyperbolic and not parabolic as in the present
setting.

We generalize such result to the parabolic setting.
Nevertheless, we are dealing with the simplest case, that
is, fixed points for the $4$-dimensional Poincar\'e map, which correspond
to
periodic orbits for the flow. In the general hyperbolic setting
usually one considers $\CCC^1$ Lambda lemmas even if for the
shadowing results only a $\CCC^0$ version is needed. The reason is
that to have $\CCC^0$ estimates including the tangential directions,
one needs to compute the $\CCC^1$ estimates at the same time. In the
present setting, since we are dealing with the invariant manifolds
of fixed points of maps, there are not tangential directions and,
therefore, one can directly compute  the $\CCC^0$ estimates \emph{\`a
la Shilnikov} (see \cite{Shilnikov67}). This is  a considerable
simplification since in the parabolic setting the classical $\CCC^1$
estimates are no longer true, as shown in~\cite{GorodetskiK12}.

\begin{lemma}\label{lem:LambdaLemma} Let $\Gamma$ be a curve which transversally
intersects $W^s(\Lambda^{[G_1,G_2]})$
at  $P\in W^s(\Lambda_{\alpha_0,G_0})$ for some
$\Lambda_{\alpha_0,G_0}\in\Lambda^{[G_1,G_2]}$. Let $Z$ be a
point on $W^u(\Lambda_{\alpha_0,G_0})$. For any
 neighborhood  $\mathcal U$ of $Z$ in $\RR^4$ and any $\varepsilon>0$, there
exists a point $a\in B_{\varepsilon}(P) \cap \Gamma$ and a positive integer $n$
(which depends
on $Z$, $\varepsilon$ and $\UU$) such that $\PP^n(a)\in\UU$.

As a consequence $W^u(\Lambda_{\alpha_0,G_0})\subset \overline{\cup_{j\geq 0} \PP^j(\Gamma)}$.
\end{lemma}

The proof of Lemma~\ref{lem:LambdaLemma} is placed in
Section~\ref{sec:proof_of_lem:LambdaLemma}. The lemma  follows from
Proposition~\ref{prop:Transition}.

\begin{remark}\label{rmk:LambdaReversed} By reversing time,
one can consider curves $\Gamma$ intersecting transversally
$W^u(\Lambda^{[G_1,G_2]})$
at  $P\in W^u(\Lambda_{\alpha_0,G_0})$. Then, we have the following statement.
Let $Z$ be a
 point on $W^s(\Lambda_{\alpha_0,G_0})$. For any
 neighborhood $\mathcal U$ of $Z$ in $\RR^4$ and any $\varepsilon>0$
there exists a point $a\in B_{\varepsilon}(P) \cap \Gamma$ and a
positive integer $n$ (which depends on $Z$, $\varepsilon$ and $\UU$)
such that $\PP^{-n}(a)\in\UU$.

As a consequence $W^s(\Lambda_{\alpha_0,G_0})\subset \overline{\cup_{j\geq 0} \PP^{-j}(\Gamma)}$.
\end{remark}


Now it only remains to apply Lemma~\ref{lem:LambdaLemma} to shadow the
transition chain given in Corollary~\ref{coro:boundedorbitscattering}.  This
argument is standard, one can see, for instance, \cite{DelshamsLS00}. We include
it here for completeness.

\begin{proposition} \label{prop:shadowing} Let $\{\Lambda_{\alpha_k,G_k}\}_{k\geq 0}$ be a family
of parabolic fixed points in~$\Lambda^{[G_1,G_2]}$ of the Poincar\'{e} map
$\PP$ in
\eqref{def:PoincareMap} such that, for all~$k$,
$W^u(\Lambda_{\alpha_k,G_k})$ intersects transversally
$W^s(\Lambda^{[G_1,G_2]})$ at $p_k\in W^s(\Lambda_{\alpha_{k+1},G_{k+1}})$.
Consider two sequences of real numbers $\{\delta_k\}_{k\geq 0}$ and
$\{\wt \delta_k\}_{k\geq 0}$, $\delta_k,\wt \delta_k >0$.
Then, there exist $a\in B_{\delta_0}(\Lambda_{\alpha_0,G_0})$
 and two sequences of natural numbers $\{N_k\}_{k\geq
0}$, $\{\wt N_k\}_{k\geq 0}$, $N_k < \wt N_k < N_{k+1} <
\wt{N}_{k+1}$ for all $k$, such that $\PP^{N_k}(a) \in
B_{\delta_k}(\Lambda_{\alpha_k,G_k})$ and $\PP^{\wt N_k}(a) \in
B_{\wt \delta_k}(p_k)$ for all $k$.
\end{proposition}

\begin{proof}
The proof follows closely the arguments in~\cite{DelshamsLS00}.

We are going to construct a sequence of nested non-empty compact
sets $\ol U_i \subset B_{\delta_0}(\Lambda_{\alpha_0,G_0})$ with
the following property: if $p \in \ol U_i$, its forward orbit
by~$\PP$ visits the balls $\ol
B_{\delta_k}(\Lambda_{\alpha_k,G_k})$ and  $B_{\tilde
\delta_k}(p_k)$ for $0 \le k \le i$.

Let $x_0 \in B_{\delta_0}(\Lambda_{\alpha_0,G_0}) \cap
W^s(\Lambda_{\alpha_0,G_0})$. We can choose an open neighborhood
$U_0$ of $x_0$ such that $U_0 \subset \ol U_0 \subset
B_{\delta_0}(\Lambda_{\alpha_0,G_0})$. By
Lemma~\ref{lem:LambdaLemma} (see Remark \ref{rmk:LambdaReversed}),
there exist a point $y_0 \in B_{\tilde \delta_0}(p_0)\cap
W^s(\Lambda_{\alpha_1,G_1})$ and an integer $n_0 >0$ such that
$\PP^{-n_0}(y_0) \in U_0$. By continuity of the map $\PP$, there
exists an open neighborhood $V_0$ of $y_0$ such that
$\overline{\PP^{-n_0}(V_0) } \subset U_0$. Since $y_0 \in
W^s(\Lambda_{\alpha_1,G_1})$, there exists $n_1>0$ such that
$\PP^{n_1} (y_0) = x_1 \in B_{\delta_1}(\Lambda_{\alpha_1,G_1}) \cap
W^s(\Lambda_{\alpha_1,G_1})$. By continuity, there exists an open
neighborhood $\tilde U_1$ of $x_1$, such that $\PP^{-n_1}(\tilde
U_1) \subset V_0$. We define $U_1 = \PP^{-n_0-n_1}(\tilde U_1)$. By
construction, $\overline{U_1}\subset U_0 \subset \overline{U_0}$ and
if $x\in \overline{U_1}$, $\PP^{n_0}(x) \in V_0 \subset B_{\tilde
\delta_0}(p_0)$ and $\PP^{n_0+n_1}(x) \in \overline{ U_1} \subset
B_{ \delta_1}(\Lambda_{\alpha_1,G_1})$. We take $N_0 =0$,
$\tilde N_0 = n_0$ and $N_1 = n_0+n_1$.

The proof follows by induction. Then, any point in $\cap_{i\geq 0} \ol U_i\neq
\emptyset$, satisfies the claim.

\end{proof}

Proof of Theorem \ref{thm:Main} is a direct consequence of this proposition.

\begin{proof}[Proof of Theorem \ref{thm:Main}]
Let $\{\Lambda_{\al_k,G_k}\}_{k\geq 1}$ be one of the bounded orbits
given in Corollary
\ref{coro:boundedorbitscattering}. Since
$\La_{\al_k,G_k}=\SSS^+(\La_{\al_{k-1},G_{k-1}})$, $k\geq 1$, the unstable
manifold of $\La_{\al_{k-1},G_{k-1}}$ intersects transversally
the stable manifold of the cylinder in a point $p_k$ in the stable
fiber of $\La_{\al_k,G_k}$.

Now we apply Proposition~\ref{prop:shadowing}.  We take $\de_k=1/k$
(or any other positive sequence with limit $0$) and $\tilde
\delta_k=\tilde \delta$ small enough so that the balls $B_{\wt\de}(p_k)$ do
not intersect the cylinder $\Lambda^{[G_1,G_2]}$. Let $N_k$ and $\tilde N_k$ be
the natural numbers and $a\in B_{\delta_0}(\Lambda_{\alpha_0,G_0})$ be the point
given by Proposition~\ref{prop:shadowing}. Then, the orbit of $a$
is oscillatory. Indeed, $\liminf_{k \rightarrow+\infty}\text{\rm
dist}( \PP^{N_k}(a),\Lambda_{\alpha_k,G_k})=0$ since
$\delta_k\rightarrow 0$ as $k\rightarrow +\infty$ and
$\limsup_{t\rightarrow+\infty}\text{\rm dist}( \PP^{\tilde
N_k}(a),\Lambda_{\alpha_k,G_k})>0$ since the balls $B_{\wt\de}(p_k)$
do not intersect $\{x=y=0\}$. Pulling back the McGehee change of
variables $r=2/x^2$ and considering the original coordinates, we obtain
oscillatory orbits.
\end{proof}

\section{A $\CCC^0$ Lambda lemma: proof of Lemma~\ref{lem:LambdaLemma}}
\label{sec:proof_of_lem:LambdaLemma}

\subsection{Local behavior close to infinity}

System~\eqref{eq:3bpInMcGehee} can be written as
 $r=2/x^2$,
\[
 \begin{split}
 \dot x&= -\frac{1}{4} x^3 y\\
  \dot y&=\frac{1}{8}G^2x^6+\pa_r
U\left(2x^{-2},\al,t\right)=-\frac{1}{4}x^4+x^6\OO_1\\
\dot \al&=  \frac{1}{4} G x^4\\
\dot G&=\pa_\al U\left(2x^{-2},\al,t\right)=\beta(\al,t)x^6+x^8\OO_1,
 \end{split}
\]
where $\beta$ is a function which is $2\pi$-periodic in its variables and
$\OO_k$ stands for $\OO(\|(x,y)\|^k)$.

To straighten the lowest order terms, we make the change
\[
 \begin{split}
q&=\frac{1}{2}(x-y)\\
p&=\frac{1}{2}(x+y)
 \end{split}
\]
and to make $\al$ a central variable, we consider the new variable
$\theta=\al+Gy$. Then, we have the new system
\begin{equation}
\label{eq:system_in_McGehee_coordinates}
 \begin{split}
 \dot q&= \frac{1}{4} (q+p)^3(q+(q+p)^3\OO_0)\\
 \dot p&= -\frac{1}{4} (q+p)^3(p+(q+p)^3\OO_0)\\
 \dot \theta&=(q+p)^6\OO_0\\
\dot G&=(q+p)^6\OO_0\\
 \dot t &=1.
 \end{split}
\end{equation}
This system is a particular case of a system of the form
\begin{equation}
\label{eq:general_system_in_McGehee_coordinates}
 \begin{split}
 \dot q&= \frac{1}{4} (q+p)^3(q+(q+p)^3\OO_0)\\
 \dot p&= -\frac{1}{4} (q+p)^3(p+(q+p)^3\OO_0)\\
 \dot z &=(q+p)^6\OO_0\\
 \dot t &=1
 \end{split}
\end{equation}
where $(q,p,z)\in \RR\times \RR \times K$, $K\subset \RR^n$ a compact set. Now
$\OO_k =
\OO(\|(q,p)\|^k)$ are $T$-periodic functions on~$t$ and the estimate is uniform
for $z\in K$.

Notice that
for any $z_0\in \RR^n$, the set
\[
\wt \Lambda_{z_0} = \{q = p = 0, \; z = z_0,\;t\in\TT\}
\]
is a periodic orbit of
system~\eqref{eq:general_system_in_McGehee_coordinates}. In this section we
study its invariant manifolds. The set $\wt \Lambda =
\cup_{z_0\in K} \wt
\Lambda_{z_0}$ is invariant. In the case case of
system~\eqref{eq:system_in_McGehee_coordinates}, when $z_0 =
(\alpha_0,G_0)$, the set $\wt \Lambda$ corresponds to the
``parabolic infinity''.

\begin{proposition}
\label{prop:local_invariant_manifolds} Consider the
system~\eqref{eq:general_system_in_McGehee_coordinates}. The set
$\wt \Lambda_{z_0}$ possesses invariant stable and unstable
manifolds, $W^u_{z_0}$ and $W^s_{z_0}$. More concretely,
\[
W^u_{z_0} = \{(q,p,z,t)\mid (p,z) = \gamma^u(q,z_0,t), \; q \in
[0,q_0),t\in\TT\}
\]
where
\begin{enumerate}
\item $\gamma^u$ is $\CCC^{\infty}$ with respect to~$q$ and analytic with
respect to $z_0$ and $t$,
\item $\pi_p \gamma^u(q,z_0,t) =
\OO(q^2)$, $\pi_{z} \gamma^u(q,z_0,t) = z_0 + \OO(q^3)$, where $\pi_p$
and $\pi_{z}$ are the corresponding projections.
\end{enumerate}
The analogous statement holds for $W^s$, as a graph over~$p$.
\end{proposition}

\begin{proof} It is an immediate consequence of
Theorem~\ref{thm:invariantmanifoldsflows} in the Appendix.
Indeed, after changing the sign of time and
introducing the new variables
\[
\tilde z = \frac{1}{q+p}(z-z_0),
\]
system~\eqref{eq:system_in_McGehee_coordinates} becomes

\begin{equation}
\label{eq:system_in_McGehee_coordinates_2}
 \begin{split}
 \dot q&= -\frac{1}{4} (q+p)^3(q+(q+p)^3\OO_0)\\
 \dot p&= \frac{1}{4} (q+p)^3(p+(q+p)^3\OO_0)\\
 \dot{\tilde z}&=\frac{1}{4}(q+p)^2(q-p) \tilde z + (q+p)^5\OO_0\\
 \dot t &=1
 \end{split}
\end{equation}
and depends analytically on $z_0$ and $t$. It satisfies the hypotheses of
Theorem~\ref{thm:invariantmanifoldsflows} with $N=4$ and $c = 1/4$.
Then, the origin of~\eqref{eq:system_in_McGehee_coordinates_2} has
an invariant stable manifold parameterized by
\[
\tilde \gamma (u,z_0,t) = (u,0,0,0) + \OO(u^2).
\]
From this parametrization, we have that the invariant manifold is
the graph of a function over~$q$ satisfying $(p,\tilde z) =
\OO(q^2)$. The claim follows simply restoring the original
variables.

The proof for the stable manifold is analogous.
\end{proof}

Once we have the existence and regularity of the invariant
manifolds, their straightening can be easily accomplished.

\begin{proposition}
\label{prop:straighteningofthemanifolds} There exists a $C^{\infty}$
change of variables which transforms
system~\eqref{eq:general_system_in_McGehee_coordinates} into
\begin{equation}\label{def:DefModel}
\begin{split}
\dot q&=f(q,p)^3q(1+\OO_1)\\
\dot p&=-f(q,p)^3p(1+\OO_1)\\
\dot z &= f(q,p)^2 q p \OO_1 \\
\dot t &=1
\end{split}
\end{equation}
defined for $q,p \ge 0$, where $f(q,p) = q+p+\OO_2$ is a
$\CCC^{\infty}$ function in $\{q,p \ge 0\}$ with bounded derivatives.
\end{proposition}

\begin{proof}
The change is the composition of two consecutive changes of
variables. The first one is defined as follows. By
Proposition~\ref{prop:local_invariant_manifolds}, the map
\[
\Gamma:(q,z_0,t) \mapsto (q,\pi_{z} \gamma^u(q,z_0,t),t)
\]
is a $\CCC^{\infty}$ diffeomorphism from $\{(q,z_0,t)\mid 0<q<\delta,t\in\TT\}$
onto its image, for some $\delta>0$. By item 2 of
Proposition~\ref{prop:local_invariant_manifolds}, there exists
$\delta'>0$ such that  $\{(q,z,t)\mid 0<q<\delta'\} \subset
\Gamma(\{(q,z_0,t)\mid 0<q<\delta\})$. Hence,
\[
(q,z_0(q,z,t),t) = \Gamma^{-1}(q,z,t) = (q, z,t) + \OO(q^3)
\]
is a diffeomorphism.

The first change that we perform is given by
\[
\tilde q = q, \quad \tilde p = p - \pi_p \gamma^u(q,z_0(q,z,t)), \quad
\tilde z = z_0(q,z,t).
\]
It is clearly a change of variables if $q$ is small, and straightens
the unstable leaves. Analogously, one straightens the stable leaves.
The second change is the identity on the unstable manifold.
\end{proof}

\subsection{The Lambda lemma}

To state the Lambda Lemma \ref{lem:LambdaLemma} it is more convenient to work
with  the Poincar\'e map
\begin{equation}\label{def:PoincareMap}
\PP: \{t=t_0\} \longrightarrow \{t=t_0+2\pi\}
\end{equation}
 associated to the flow of the system~\eqref{def:DefModel}.

Lemma \ref{lem:LambdaLemma} follows from the
next statement.

\begin{proposition}\label{prop:Transition} Consider a point
$Q=(q_f,0,z_0)\in W^u(0,0,z_0)$ and  a $\CCC^1$ curve $C$
parameterized by $\al(\de)=(\delta, \al_p(\de), \al_z(\delta))$, $\de\in
[0,\de_0]$ for some $\delta_0>0$, which is transversal to $\{q=0\}$
in a point $P=\al(0)=(0, p_0, z_0)$. For any $\rr\ll
1$, there exists a sequence $\{\de_k\}_{k\geq 1}\searrow 0$ as
$k\rightarrow \infty$ with $0<\de_k<\de_0$, and an increasing
sequence of $\{n_k \}_{k\geq 1}\subset \NN$, $n_k\rightarrow+\infty$
as $k\rightarrow+\infty$, such that
\[
 \PP^{n_k}(\al(\de_k))\in B_\rr(Q).
\]
\end{proposition}

\begin{proof}
We first study the  associated
equation \eqref{def:DefModel} and later we show how to deduce the statement for
the
 Poincar\'e map.

As $\{q=0\}$ and $\{p=0\}$ are invariant manifolds for system
\eqref{def:DefModel}, the region $\{q>0,p>0\}$ is invariant. In this region,
by Proposition \ref{prop:straighteningofthemanifolds}, $f$  is strictly positive
for $(q,p)$ small enough and satisfies $|f(q,p)|\ge K \|(q,p)\|$, for some $K>0$
depending only on the
domain.

We  rescale time by setting $\frac{ds}{dt}=f(q,p)^3$.  In the new
time $s$, in the region  $\{q>0,p>0\}$, the system becomes
\begin{equation}\label{ModResc}
\begin{split}
\dot q&=q(1+\mathcal{O}_1)\\
\dot p&=-p(1+\mathcal{O}_1)\\
\dot z &= q p \OO_0\\
\dot t &=(f(q,p))^{-3}.
\end{split}
\end{equation}
Call $\Psi^s$ the flow associated with this equation.

Fix $\eps_0>0$ and $\zeta_0 >4$ such that the equation
\eqref{ModResc} is well defined for
\[
 (q,p,z,t)\in D= (0,\eps_0]^2\times
B_{\zeta_0}(z_0)\times\TT\subset\RR^2\times \RR^n\times\TT.
\]
There  exists $K>0$ such that  the terms $\OO_0$ and $\OO_1$ appearing in
\eqref{ModResc} satisfy $|\OO_0|\leq K$ and $|\OO_1|\leq
K\|(q,p)\|\ll 1$. Choose $\eps\in (0,\eps_0)$ such that
$\wt\eps=K\eps\in (0,1/10)$. Then, for any point in
$(0,\eps]^2\times [G_1,G_2]\times\TT^2$,
\begin{equation}\label{eq:RescaledEqEstimates}
\begin{array}{cccc}
(1-\tilde\varepsilon)q&\leq \dot q\leq &(1+\tilde\varepsilon)q & \\
-(1+\tilde\varepsilon)p&\leq \dot p\leq &-(1-\tilde\varepsilon)p& \\
- K q p &\leq \dot z_i \leq& K q p,&  \quad i =1,\dots,n
\end{array}
\end{equation}
where $z= (z_1,\dots,z_n)$.

The points $P$ and $Q$ introduced in the statement of the
proposition can be chosen  such that $0<q_f,p_0<\eps/10$. Fix
$\wt\delta_0>0$ with $\wt \delta_0<\min\{\delta_0, \varepsilon/10\}$.
Fix~$\rho>0$ and define $\mathcal{U}=B_\rho(Q)\subset\RR^4$. For any
$0<\de<\wt\delta_0$, there exists $T^*=T^*(\varepsilon,\delta)>0$
such for any time in $[0,T^*]$, the orbit  of the point~$\al(\de)$
under the flow~$\Psi^s$ of system~\eqref{ModResc}  does not leave
the domain~$D$.

Applying Gronwall estimates to the first two equations of
\eqref{eq:RescaledEqEstimates} with initial condition $(q(0),p(0),
z(0))=\al(\de)$
and $s\in [0,T^*]$, one obtains the following inequalities
\begin{equation}\label{q(s)p(s)}
\begin{array}{ccc}
q(0)e^{(1-\tilde\varepsilon)s}&\leq  q(s) \leq &q(0)e^{(1+\tilde\varepsilon)s}\\
p(0)e^{-(1+\tilde\varepsilon)s}&\leq  p(s) \leq
&p(0)e^{-(1-\tilde\varepsilon)s}.
\end{array}
\end{equation}
Moreover,
\begin{equation}\label{eq:GandTheta}
\|z(s)-z(0)\|\leq \left\|\int_0^s q p \OO_0\right\|\leq
5Ke^{s/5} q(0)p(0).
\end{equation}
It is clear that the variable that leaves first the domain $D$ is
$q$ and therefore $T^*$ satisfies
\[
 \frac{1}{1+\wt\eps}\ln\frac{\eps}{\de}\leq
T^*\leq\frac{1}{1-\wt\eps}\ln\frac{\eps}{\de}.
\]
Since $q_f\leq \eps/10$, by continuity, for $\de\leq \wt\de_0$ there
exists $S=S(\de)\leq T^*$ such that
$\Psi^S(\al(\de),0)\in\{q=q_f\}$. Moreover, by the
Gronwall
estimates in \eqref{q(s)p(s)},
\begin{equation}\label{def:RecaledTimeTransition}
  \frac{1}{1+\wt\eps}\ln\frac{q_f}{\de}\leq
S\leq\frac{1}{1-\wt\eps}\ln\frac{q_f}{\de}.
\end{equation}
For such $S$, using the second inequality in \eqref{q(s)p(s)} we get
that
\[
p(S)\leq p(0)e^{\frac{1-
\wt\eps}{1+\wt\eps}\log\frac{\delta}{q_f}}\leq
p(0)\left(\frac{\delta}{q_f}\right)^{\frac{1-
\tilde\varepsilon}{1+\tilde\varepsilon}}.
\]
Note that, since $\wt\eps\ll1$,  the right hand side converges to zero with
$\delta$.
Now we use \eqref{eq:GandTheta}, to bound $\|z(S)-z(0)\|$.  We have
that
\[
 \begin{split}
  \|z(S)-z_0\|&\leq \|z(S)-z(0)\| +\|z(0)-z_0\|\\
 &\leq 5Ke^{ S/5} q(0)p(0)+\|\al\|_{\CCC^1}\de\\
&\leq 5K
p(0)q_f^{\frac{1}{5(1-\tilde\varepsilon)}}\delta^{1-\frac{1}{5(
1-\tilde\varepsilon)}}+\|\al\|_{\CCC^1}
\de
 \end{split}
\]
where we have used that $q(0)=\de$.

We want the final point in $\{q=q_f\}$ to belong to $B_\rr(Q)$. We
choose $\de$ such that this is true. We need
\[
 |p(S)-p_0|\leq\frac{\rr}{4},\quad \text{and} \quad \|z(S)-z_0\|\leq\frac{\rr}{4}.
\]
Then, we need to choose $\de$ such that
\[
 \|\ga\|_{\CCC^0}\left(\frac{\delta}{q_f}\right)^{\frac{1-
\tilde\varepsilon}{1+\tilde\varepsilon}}\leq\frac{\rr}{4}\,\, \text{
and }\,\,
5K
\|\ga\|_{\CCC^0}q_f^{\frac{1}{5(1-\tilde\varepsilon)}}\delta^{1-\frac{
1}{5(1-\tilde\varepsilon)}}+\|\ga\|_{\CCC^1}\de\leq
\frac{\rho}{4}.
\]


All is left to prove now is that, in the original time variable
(before rescaling), that is for the flow associated to the equation
\eqref{def:DefModel}, the time $T$ needed for the transition from
$\ga(\de)$ to $B_\rr(Q)$ can be chosen as a multiple of $2\pi$. This
is achieved by taking $\de$ small enough and allows us to say that
what we obtained makes sense for the Poincar\'e map. The argument
goes as follows. Fix the starting section $p=p(0)$ and the final
final section $q=q_f$. The transition time $S$  in the rescaled time
depends continuously on $\de$ and, by
\eqref{def:RecaledTimeTransition}, satisfies that
$S(\delta)\rightarrow\infty$ as $\delta \rightarrow 0$. Since the
orbit never leaves the domain $D$, we have that $dt/ds\geq
1/\varepsilon^3$. Therefore, the transition time
$T(\delta)\rightarrow \infty$ as $\delta\rightarrow 0$. By
continuity, it must pass through a multiple of $2\pi$. In fact, it
must pass through a multiple of $2\pi$ infinitely many times
$T(\delta_k)$ with $\delta_k\rightarrow 0$ as stated in the
proposition.
\end{proof}

\section*{Acknowledgements:}
M.G., P. M. and T. S. are partially supported by the  Spanish
MINECO-FEDER Grant MTM2012-31714 and the Catalan Grant 2014SGR504. L. S. is
partially supported by the EPSRC grant EP/J003948/1.

\appendix

\section{Invariant manifolds of parabolic points and dependence with
respect to parameters} \label{sec:invman}

Here we present a version of Theorem~2.1 in~\cite{BaldomaFdlLM07} on
the invariant manifolds of parabolic fixed points where we also include
their dependence with respect to parameters. We need this statement
for Proposition~\ref{prop:local_invariant_manifolds}, which provides
the proper set of coordinates in which we derive de Lambda Lemma.
Here we  deal with the analytic case, which is the one relevant
in the present work, unlike the setting of~\cite{BaldomaFdlLM07},
where $\CCC^k$ maps were considered. The proof of the theorem we
present here follows the same lines of the one
in~\cite{BaldomaFdlLM07}, but is simpler and shorter. The main difference is
that, to deal with analyticity, we work with complex domains.

We first present the statement for maps and then we deduce the
analogous result for time periodic flows.

Given a local diffeomorphism $F$ from a neighborhood of $(0,0) \in
\RR\times \RR^n$ to $\RR\times \RR^n$ such that $F(0,0) = (0,0)$, an
open convex set $\widetilde V$ such that $(0,0) \in \partial
\widetilde V$ and $r>0$, we define the \emph{local stable set of
$(0,0)$ with respect to $\widetilde V$} as
\[
W^s_{\widetilde V,r} = \{(x,y)\in
\widetilde V \cap B_r \mid F^k(x,y) \in \widetilde V \cap B_r, \;
\text{for all $k\ge 0$}, \; \lim_{k\to \infty} F^k(x,y) = (0,0)\},
\]
where $B_r$ denotes the ball of radius~$r$ in $\RR\times \RR^n$.

\begin{theorem} \label{thm:invariantmanifoldsmaps}
Let $U \subset \RR\times \RR^n$ be a neighborhood of the origin
$(x,y)=(0,0)$, where $x\in\RR$ and $y\in \RR^n$, $\AAA \subset
\RR^p$ an open set and $F=(F^1, F^2):U \times \AAA \rightarrow
\RR\times \RR^n $ a real analytic map such that its derivatives up
to order $N\ge 2$ at the origin are independent of the parameters~$\lambda\in
\AAA$, $F(0,0,\lambda) = 0$, $DF(0,0,\lambda) = \Id$
\[
D^jF(0,0,\lambda) = 0, \qquad \mbox{ for } \quad 2\le j\le N-1
\]
and
\[
\frac{\partial ^N F^1 }{\partial x^N}(0,0,\lambda)
=: -c <0, \qquad \frac{\partial ^N F^2}{\partial x^N}(0,0,\lambda) =
0
\]
and
\[
\text{\rm spec}\, \frac{\partial^N F^2}{\partial
x^{N-1}
\partial y}(0,0,\lambda) \subset \{z\in \CC \mid \Re z > 0\}.
\]

Then there exists $t_0>0$ and a $\CCC^{\infty}$ map $K: [0,t_0)\times
\AAA \subset \RR \times \RR^p \rightarrow \RR\times \RR^n$ of the
form $K(t) = (t,0) + O(t^2)$, analytic in $(0,t_0)\times \AAA$,
and a polynomial $R(t) = t-c t^N + \tilde c (\lambda) t^{2N-1}$,
with $\tilde c$ analytic in $\AAA$, such that
\begin{equation} \label{mainformula}
F\circ K = K\circ R.
\end{equation}
Furthermore, there exists an open convex set $\widetilde V$, with
$(0,0) \in \partial \widetilde V$, containing the line $\{x>0, y=
0\}$, $r>0$ and $t_0$ such that the  range of $K$ is
$W^s_{\widetilde V,r}$.
\end{theorem}

\begin{proof}
Along the proof we  skip the dependence of the functions
on~$\lambda$.

First of all, we remark that the hypotheses of
Theorem~\ref{thm:invariantmanifoldsmaps} imply those of Theorem~3.1
in~\cite{BaldomaF04}. Hence, there exists an open convex set
$\widetilde V$, with $(0,0) \in \partial \widetilde V$, containing
the line $\{x>0, y= 0\}$ and $r>0$ such that the local stable set
$W^s_{\widetilde V,r}$ is the graph of a Lipschitz function $\psi:
[0,s_0) \to \RR^n$. The fact that the set $\widetilde V$ can be
chosen independently of~$\lambda$ follows from the estimates of
Theorem~3.1 in~\cite{BaldomaF04}, where it is proven that it
contains a cone whose size only depends on the derivatives of $F$ up
to order~$N$. Then, if $K$ is the solution of~\eqref{mainformula}
given by the theorem, since $K(t,0) = (t,0) + O(t^2)$, its range is
the graph of a Lipschitz function, which must coincide with
$W^s_{\widetilde V,r}$. This proves the last statement of the
theorem, assuming that the previous ones are true.

Now we prove the existence and properties of such a~$K$.

Under the current hypotheses, we can apply Lemma~3.1
of~\cite{BaldomaFdlLM07} and we obtain that for any $k\ge N$ there
exist polynomials $K^{\le k} : \RR \to \RR^{1+n}$, of degree~$k$,
and  $R(t) = t-c t^N + \tilde c (\lambda) t^{2N-1}$ such that
\begin{equation}\label{eq:formal}
F\circ K^{\le k}(t) - K^{\le k}\circ R (t) = \OO(t^{k+N}).
\end{equation}
It is straightforward from the proof of Lemma~3.1 that the
coefficients of $K^{\le k}$ and $R$ are analytic functions of the
coefficients of the Taylor expansion of $F$ at $(0,0)$ up to
order~$k$. Hence, they depend analytically on~$\lambda$.

We claim
\begin{lemma}
\label{lem:Vinvariant} Let $\alpha = 1/(N-1)$. Let $\widehat
\AAA$ be a neighborhood of $\AAA$ in $\CC^p$ in which the
map~$F$ is analytic. There exists $a_0>0$ such that for any $0 < a <
a_0$, there exists $\rho'
>0$ such that the set $V = \{t \in \CC\mid |\arg t | \le a, \; 0<|t|
\le \rho\}$ satisfies $R(V) \subset V$, for any $0<\rho<\rho'$ and
$\lambda \in \widehat \AAA$.
\end{lemma}
\begin{proof}[Proof of Lemma~\ref{lem:Vinvariant}]
We remark that, for $t\in V$,
\[
\begin{aligned}
\arg (1-ct^{N-1} + \tilde c t^{2N-2}) & = - i \log \frac{(1-c t^{N-1}
+ \tilde c t^{2N-2})}{|(1-c t^{N-1} + \tilde c t^{2N-2})|} \\
 & = -c |t|^{N-1} \sin((N-1) \arg t) + \OO(|t|^{2N-2}) \\
 & = -(N-1) c|t|^{N-1} \arg t + |t|^{N-1} \OO(a^2) +
 \OO(|t|^{2N-2}).
\end{aligned}
\]
If $\rho$ is such that $1 -(N-1)
c\rho^{N-1}>0$, for some   $M>0$  we have
\[
\begin{aligned}
|\arg R(t)|& = |\arg (t(1-c t^{N-1} + \tilde c t^{2N-2}))| \\
& =  |\arg t + \arg (1-c t^{N-1} + \tilde c t^{2N-2})| \\
& = |(1 -(N-1) c|t|^{N-1}) \arg t + |t|^{N-1} \OO(a^2) +
 \OO(|t|^{2N-2})| \\
 & \le (1 -(N-1) c|t|^{N-1}) |\arg t| + M |t|^{N-1} a^2 +
 M|t|^{2N-2}.
\end{aligned}
\]
From the last inequality we obtain that, if $\rr$ satisfies,
\[
\rho^{N-1}<(N-1)c a\left(1-\frac{Ma}{(N-1)c}\right),
\]
then $|\arg R(t) | < a$, which proves the claim.
\end{proof}

We choose the constants $a$ and $\rho$ accordingly to
Lemma~\ref{lem:Vinvariant}. Taking $a$ and $\rho$ smaller if
necessary, it is clear that there exists $0< b < c < d$ such that
for all $t\in V$ and any $\lambda \in\widehat \AAA$,
\begin{equation}
\label{bound:contraction} R_d(|t|):= |t|-d|t|^N \le |R(t)| \le |t| -
b|t|^N =:R_b(|t|).
\end{equation}
Notice that $b$ and $d$ can be chosen arbitrarily close to $c$
taking $a$ and $\rho$ small enough.

Next lemma describes the contraction provided by the nonlinear
terms.

\begin{lemma}
\label{lem:contraction} Let $b_0, d_0$ and $s>0$ such that
$b_0^{N-1} = \alpha b^{-1}$, $d_0^{N-1} = \alpha d^{-1}$ and $b_0
s^{-\alpha} = \rho$. Then there exist two sequences, $(b_i)_{i\ge
0}$ and $(d_i)_{i\ge 0}$ such that for any $0<\beta< 1$,
\begin{equation}
\label{eq:limtsbkqk} b_i = b_0(1+\OO(i^{-\beta})),\quad d_i =
d_0(1+\OO(i^{-\beta})), \quad i \ge 0
\end{equation}
for $i\ge 0$,
\begin{equation}
\label{eq:bkdk} \frac{d_{i+1}}{(s+i+1)^{\alpha}} <
\frac{b_{i}}{(s+i)^{\alpha}},
\end{equation}
and
\begin{equation}
\label{eq:Rbkdk} R_b\left(\frac{b_{i}}{(s+i)^{\alpha}} \right) =
\frac{b_{i+1}}{(s+i+1)^{\alpha}},\quad
R_d\left(\frac{d_{i}}{(s+i)^{\alpha}} \right) =
\frac{d_{i+1}}{(s+i+1)^{\alpha}}.
\end{equation}
 Furthermore, the sets
\begin{equation}
\label{def:Vk} V_i = \left\{t\in V\mid
\frac{d_{i+1}}{(s+i+1)^{\alpha}} \le |t |  \le
\frac{b_{i}}{(s+i)^{\alpha}} \right\}
\end{equation}
satisfy  $V = \cup_{i\ge 0} V_i$ and $R(V_i) \subset V_{i+1}$.
Consequently, if $t\in V_i$, for any $j\ge 0$
\begin{equation}
\label{eq:contraction}
\frac{d_{0}(1+\OO(i^{-\beta}))}{(s+i+j+1)^{\alpha}} \le |R^j(t) |
\le \frac{b_{0}(1+\OO(i^{-\beta}))}{(s+i+j)^{\alpha}}.
\end{equation}
\end{lemma}
\begin{proof}[Proof of Lemma~\ref{lem:contraction}]
The two relations in~\eqref{eq:Rbkdk} define the numbers $b_k$ and
$d_k$ for $k\ge 1$. To prove~\eqref{eq:bkdk} we proceed by
induction. For the first step, since $b < d$, we have that $d_0 <
b_0$. Then, using~\eqref{bound:contraction}, the fact that $R_d(r),
R_b(r) < r$, for $0< r < r_0$ and are strictly increasing,
\[
\frac{d_1}{(s+1)^{\alpha}} = R_d\left( \frac{d_0}{s^{\alpha}}\right)
\le R_b\left( \frac{d_0}{s^{\alpha}}\right) < R_b\left(
\frac{b_0}{s^{\alpha}}\right) < \frac{b_0}{s^{\alpha}}.
\]
Since
\[
R_b\left(\frac{b_0}{s^{\alpha}}\right)= \frac{b_1}{(s+1)^{\alpha}},
\]
we have that $d_1 < b_1$ and we can perform the induction procedure.

Relations~\eqref{eq:limtsbkqk} follow from~Lemma~4.4 in~\cite{BaldomaFdlLM07}.

From~\eqref{eq:limtsbkqk} and~\eqref{eq:bkdk} follows that $V =
\cup_{k\ge 0} V_k$. Inclusion  $R(V_k) \subset V_{k+1}$ follows
from~\eqref{eq:Rbkdk} and inequality~\eqref{bound:contraction}.
\end{proof}

We choose $k>2N-1$.  Let
$K^{\le}= K^{\le k} : \RR \to \RR^{1+n}$, polynomial of degree~$k$ given in
\eqref{eq:formal}, and $R(t) =
t-c t^N + \tilde c (\lambda) t^{2N-1}$ be such that
\begin{equation}
\label{eq:errork} E(t):= F\circ K^{\le }(t) - K^{\le }\circ R (t) =
\OO(t^{k+N}).
\end{equation}
Since from now on $k$ is fixed, we skip the dependence on~$k$ of
$K^{\leq}$ and $E$. It is worth to remark that if one chooses $k'>k$
and finds the corresponding $K^{\leq}$ of degree~$k'$, its terms up
to degree~$k$ coincide with the former (since we keep~$R$ fixed in
the construction).

We want to find a solution of
\[
F\circ (K^{\le } + \varphi )- (K^{\le }+
\varphi) \circ R  =0.
\]
We rewrite this equation as
\begin{equation}
\label{eqK>lin} \LL(\varphi)
 = \FF(\varphi),
\end{equation}
where
\begin{equation}
\label{eq:linearoperator} \LL(\varphi) = (DF \circ K^{\le})\varphi -
\varphi \circ R
\end{equation}
and
\begin{equation}
\label{eq:nonlinearpart} \FF(\varphi) = -  E -F \circ
(K^{\le}+\varphi)+F \circ K^{\le} + (DF\circ K^{\le} )\varphi.
\end{equation}
The operator defined by~\eqref{eq:linearoperator} is linear. A
formal inverse is given by the formula
\begin{equation}
\label{def:G} \GG(\psi) = \sum_{j\ge 0} \left(\prod_{i=0}^j
(DF)^{-1} \circ K^{\le} \circ R^i\right) \psi \circ R^j.
\end{equation}
(see~\cite{BaldomaFdlLM07}).

As explained at the beginning of Section~4.4
in~\cite{BaldomaFdlLM07}, we make the rescaling
 $\ol y = \delta y$ for some parameter $\de$ and from now on we work in this
rescaled variable without changing the notation for $K^\leq$ and $F$.  For any
$\sigma>0$, choosing appropriately the norm in $\RR\times \RR^{n}$
and taking $a$, $\rho$ and $\delta$ small enough, we can assume that, for
some $C>0$,
\begin{equation}
\label{bound:inverseF} \|DF^{-1}(K^{\le}(t))\| \le 1+(N+\sigma) d
|t|^{N-1} + C|t|^N, \quad \text{for all $t \in V$.}
\end{equation}

\begin{lemma} \label{expandingterms} Let $\{V_i\}_{i\ge 0}$ be the family of
sets defined
in~\eqref{def:Vk} and $s>0$ the number defined in
Lemma~\ref{lem:contraction}. There exists $C>0$ such that for any
$l\ge 0$, $t\in V_l$, the following inequalities hold
\[
\left\|\prod_{i=0}^j (DF)^{-1} \circ K^{\le} \circ R^i (t) \right\|
\le C \left( \frac{s+l+j}{s+l}\right)^{(N+\sigma) \alpha d b^{-1}}
\]
\end{lemma}

\begin{proof}[Proof of Lemma~\ref{expandingterms}] By
Lemma~\ref{lem:contraction}, $R^i(V_l) \subset V_{l+i}$.
By~\eqref{bound:inverseF}, and using the definition and properties of
the sets $V_l$ in Lemma~\ref{lem:contraction}, if $t\in V_l$,
\[
\begin{aligned}
\left\| (DF)^{-1} \circ K^{\le} \circ R^i (t) \right\| & \le
1+(N+\sigma) d \left(\frac{b_{i}}{(s+i+l)^{\alpha}}\right)^{N-1} +
C\left(\frac{b_{i}}{(s+i+l)^{\alpha}}\right)^N \\
& \le 1+ \frac{(N+\sigma) \alpha d b^{-1}}{s+i+l} +
\frac{C}{(s+i+l)^{1+\gamma}},
\end{aligned}
\]
where $\gamma = \min\{\alpha,\beta\}$ (see \eqref{eq:limtsbkqk}). Then, redefining $C$,
\[
\begin{aligned}
\left\|\prod_{i=0}^j (DF)^{-1} ( K^{\le} ( R^i (t))) \right\| & \le
\exp \left(  \sum_{i=0}^j \log \left( 1+ \frac{(N+\sigma) \alpha d
b^{-1}}{s+i+l} + \frac{C}{(s+i+l)^{1+\gamma}} \right)
\right) \\
& \le \exp \left(  \sum_{i=0}^j  \left( \frac{(N+\sigma) \alpha d
b^{-1}}{s+i+l} + \frac{C}{(s+i+l)^{1+\gamma}} \right) \right) \\
& \le \exp \left(  \log \left(\frac{s+j+l}{s+l-1}\right)^{(N+\sigma)
\alpha d b^{-1}} + \frac{C}{\gamma}\left(
\frac{1}{(s+l-1)^{\gamma}}-\frac{1}{(s+l+j)^{\gamma}}
\right)\right),
\end{aligned}
\]
which implies the claim with a suitable $C$.
\end{proof}

In order to solve equation~\eqref{eqK>lin}, we introduce the space
\[
\XX_m = \{\varphi:V\times \wh \AAA \to \CC^{1+n}\mid
\text{$\varphi$ analytic},\; \sup_{t\in V,\lambda \in\wh \AAA}
|t^{-m} \varphi(t,\lambda)| < \infty\},
\]
which is a Banach space with the norm
\[
\|\varphi\|_m = \sup_{t\in V,\lambda \in\wh \AAA} |t^{-m}
\varphi(t,\lambda)|.
\]
By \eqref{eq:errork}, the function~$E$ in~\eqref{eq:errork} belongs to $\XX_{k+N}$.

\begin{lemma}
\label{lem:Gbounded} Assume $m \in \NN$ satisfies $m > N+(N+\sigma)
\alpha d b^{-1}+2$. Let $\FF$ and $\GG$ be the operators defined in~\eqref{eq:linearoperator} and~\eqref{def:G}.
Then
$\GG:\XX_{m+N-1} \to \XX_{m}$ is linear,
bounded and $\LL \circ \GG = \Id$ on $\XX_{m+M-1}$.
\end{lemma}
\begin{proof}[Proof of Lemma~\ref{lem:Gbounded}]
%
Let $\psi \in \XX_{m+N-1}$. For any $l\ge 0$, $t\in V_l$, by
Lemma~\ref{expandingterms}, and inequalities~\eqref{eq:contraction},
we have that for some constant $C$ independent of $\psi$ and $l$,
\[
\begin{aligned}
|\GG(\psi)(t)| & \le \sum_{j\ge 0} \left\|\prod_{i=0}^j
(DF)^{-1} \circ K^{\le} \circ R^i (t)\right\| \left| \psi \circ R^j (t)\right| \\
& \le  C  \sum_{j\ge 0} \left( \frac{s+l+j}{s+l}\right)^{(N+\sigma)
\alpha d b^{-1}} \left(
\frac{b_{0}(1+\OO(l^{-\beta}))}{(s+l+j)^{\alpha}}\right)^{m+N-1}
\|\psi\|_{m+N-1} \\
& \le \frac{C}{(s+l)^{\alpha (m+N-1) -1}}\|\psi\|_{m+N-1}
\end{aligned}
\]
Hence, since $V = \cup_{l\ge 0} V_l$,
\[
\|\GG(\psi)\|_{m}  = \sup_{l\ge 0 }\sup_{t\in V_l,\lambda \in\wh
\AAA} |t^{-m} \GG(\psi)(t)|  \le C \sup_{l\ge 0 }
\frac{(s+l+1)^{\alpha m}}{(s+l)^{\alpha (m+N-1) -1}}\|\psi\|_{m+N-1}
 \le C \|\psi\|_{m+N-1}.
\]
The above calculations show that the sums that define $\GG$ are
absolutely convergent and can be reordered if $\psi \in
\XX_{m+N-1}$. Then a simple computation shows that $\LL \circ \GG =
\Id$
\end{proof}

To complete the proof it is enough to check that the operator $\FF$
defined in~\eqref{eq:nonlinearpart} is well defined and Lipschitz
between appropriate Banach spaces with Lipschitz constant small
enough. Recall that $\FF(0) = -E \in \XX_{k+N}$.
\begin{lemma}
\label{lem:F} Let $r>0$ be such that the map $F$ is analytic  and
bounded at $B_r(0,0) \times \AAA\subset \RR\times \RR^n\times
\RR^p$, where $B_r(0,0)$ is the ball of radius $r$ centered at the
origin in $\RR\times\RR^n$. Let $\BB\subset \XX_{k+1}$ the ball of
radius~$R>0$. Assume that $\rho< (r/R)^{1/(k+1)}$. Then $\FF:\BB \to
\XX_{k+N}$ is well defined and Lipschitz with
\[
\text{\rm lip}\; \FF \le C R \rho^{k-N+1}.
\]
for some $C>0$ independent of $R$ and $\rho$.
\end{lemma}

\begin{proof}[Proof of Lemma~\ref{lem:F}]
Let $\varphi\in \BB$. It satisfies $|\varphi(t)| \le
\|\varphi\|_{k+1} |t|^{k+1} \le R \rho^{k+1} <r$. Hence, the
function $\FF(\varphi)$ is well defined and
\[
\|\FF(\varphi)\|_{k+N} \le \|E\|_{k+N}+\sup_{t\in V, \lambda \in
\AAA} C |t|^{-k-N} |\varphi(t)|^2 \le  \|E\|_{k+N}+\sup_{t\in V,
\lambda \in \AAA} C |t|^{k-N+1} \|\varphi\|_{k+1}^2 < \infty.
\]
As for the Lipschitz constant of~$\FF$, for any $\varphi,\varphi'
\in \BB$, since $k>2N-1$,
\[
\begin{aligned}
\|\FF(\varphi)-\FF(\varphi')\|_{k+N} & \le \sup_{t\in V, \lambda \in
\AAA} C |t|^{-k-N}
\max\{|\varphi(t)|,|\varphi'(t)|\}|(\varphi-\varphi')(t)|\\
& \le \sup_{t\in V, \lambda \in \AAA} C |t|^{k-N+1}
R\|\varphi-\varphi'\|_{k+1} \\
& \le CR \rho^{k-N+1} \|\varphi-\varphi'\|_{k+1}.
\end{aligned}
\]
\end{proof}

We consider the fixed point equation
\begin{equation}
\label{eq:fixedpoint} \varphi = \GG \circ \FF( \varphi).
\end{equation}
Let  $k > \min\{ 2N-1,N+(N+\sigma) \alpha d
b^{-1}+1\}$ and choose  $m=k+1$ in Lemma~\ref{lem:Gbounded}. In Lemma~\ref{lem:F}, take $R = 2
\|\GG\|\|E\|_{k+N}$ and $\rho <\min\{(r/R)^{1/(k+1)},
(R\|\GG\|)^{-1/(k-N+1)}\}$. Then, the map $\GG
\circ \FF :\BB \to \BB$ is well defined and Lipschitz with
$\text{\rm lip}\,( \GG \circ \FF) <1$. Hence,
equation~\eqref{eq:fixedpoint} has a unique solution $\varphi^* \in
\BB\subset \XX_{k+1}$. By Lemma~\ref{lem:Gbounded},
$\varphi^*$ is a solution of equation~\eqref{eqK>lin}.

Up to this point we have found a solution of
equation~\eqref{mainformula} of the form $K^{\le} + \varphi^*$.
Since $K^{\le}$ is a polynomial of degree~$k$, it is $\CCC^{k}$ at the
origin. Also, since $\varphi^* \in \XX_{k+1}$, $\lim_{t\to 0} D^j
\varphi^*(t) = 0$, for $0\le j \le k$. Hence, $\varphi^*$ is also
$\CCC^{k}$ at the origin. Since $k$ is arbitrary and a solution of the
equation~\eqref{eq:fixedpoint} for $k'> k$ also provides
(conveniently rewritten) a solution of the equation for~$k$, taking
$\rho$ smaller, if necessary, we have that $K$ is $\CCC^{\infty}$, for
real $t$, at the origin.

This completes the proof of Theorem~\ref{thm:invariantmanifoldsmaps}.

\end{proof}

In what follows, given $T>0$, $\TT$ stands for the torus $\RR/T$.

Given a local $T$-periodic vector field $X$ from a neighborhood of
$\{(0,0,t),\;t\in\TT\} \subset \RR\times \RR^n\times \TT$ to
$\RR\times \RR^n$ such that $X(0,0,t) = (0,0)$ for all $t\in \TT$,
an open convex set $\widetilde V \subset \RR\times \RR^n$ such that
$(0,0) \in
\partial \widetilde V$ and $r>0$, we define de \emph{local stable
set of $(0,0)$ with respect to $\widetilde V$} as
\[
W^s_{\widetilde V,r} = \{(x,y)\in
\widetilde V \cap B_r \mid \phi_{t,\tau}(x,y) \in \widetilde V \cap
B_r, \; \text{for all $t\ge 0$}, \; \lim_{t\to \infty}
\phi_{t,\tau}(x,y) = (0,0),\; \tau \in \TT\},
\]
where $B_r$ denotes the ball of radius~$r$ in $\RR\times \RR^n$ and
$\phi_{t,\tau}$ is the flow of the vector field~$X$, that is,
\[
\frac{d}{dt} \phi_{t,\tau}(x,y) = X(\phi_{t,\tau}(x,y),t),\quad
\phi_{\tau,\tau}(x,y) = (x,y).
\]

\begin{theorem} \label{thm:invariantmanifoldsflows}
Let $U \subset \RR\times \RR^n$ be a neighborhood of the origin
$(x,y)=(0,0)$, where $x\in\RR$ and $y\in \RR^n$, $\AAA \subset
\RR^p$ an open set and $X=(X^1, X^2):U \times \TT \times \AAA
\rightarrow \RR\times \RR^n $ a real analytic $T$-periodic vector field,
such that its derivatives up to order $N\ge 2$ are independent of
the parameters~$\lambda\in \AAA$ and $t$ at $(x,y)=(0,0)$, $X(0,0,t,\lambda) = 0$
\[
D^j X(0,0,t,\lambda) = 0, \qquad \mbox{ for } \quad 1\le j\le N-1
\]
and
\[
\frac{\partial ^N X^1 }{\partial
x^N}(0,0,t,\lambda) =: -c <0, \qquad \frac{\partial ^N X^2}{\partial
x^N}(0,0,t,\lambda) = 0,
\]
and
\[
\text{\rm spec}\, \frac{\partial^N
X^2}{\partial x^{N-1}
\partial y}(0,0,t,\lambda) \subset \{z\in \CC \mid \Re z > 0\}.
\]

Then there exists an open convex set $\widetilde V$, with $(0,0) \in
\partial \widetilde V$, containing the line $\{x>0, y= 0\}$, $r>0$
and a $\CCC^{\infty}$ map $K: [0,s_0)\times \TT \times \AAA \subset
\RR \times \TT \times \RR^p \rightarrow \RR\times \RR^n$ of the form
$K(s,\tau,\lambda) = (s,0) + O(s^2)$, analytic in $(0,s_0)\times
\AAA$ such that the  range of $K$ is $W^s_{\widetilde V,r}$.
\end{theorem}

\begin{proof}
Let $F_{\tau}$ be the Poincar\'{e} map of the vector field~$X$ associated to
the section $\{t = \tau\}$. It depends analytically on $\tau$ and
$\lambda$. A simple computation shows that it satisfies the
hypotheses of Theorem~\ref{thm:invariantmanifoldsmaps}. Let
$\widetilde K (s,\lambda,\tau)$ be the solution of the invariance
equation
\[
F_{\tau} \circ \widetilde K (s,\lambda,\tau) = \widetilde K
(R(s,\lambda,\tau),\lambda,\tau)
\]
provided by Theorem~\ref{thm:invariantmanifoldsmaps}. Its range is
precisely $W^s_{\widetilde V,r}$.
\end{proof}

\section{Formulas for the scattering maps of the circular
problem}\label{app:scattering}
In this section we obtain formulas for the scattering map of the
circular problem.
As the scattering maps $\wt \SSS ^\pm _0$ are defined in the extended phase space and send $s$ to $s$
 \cite{DelshamsLS08}, they can be restricted to the section $\Sigma$ (see
\eqref{def:Poincare_map}) giving rise to symplectic maps $ \SSS ^\pm _0$
 on the cylinder $\Lambda^{[G_1,+\infty]}$. Recall that the result in  \cite{DelshamsLS08} deals
with a normally hyperbolic cylinder, whereas
 in our case the manifold $\Lambda$ is a ``normally parabolic invariant cylinder''
with stable and unstable manifolds.
 The arguments in \cite{DelshamsKRS14} extend the proof to this case.
 Moreover, as it is explained
in \cite{DelshamsKRS14}, the conservation of the Jacobi constant implies that
$G$ is a first integral of the two scattering maps.

Since the scattering maps $ \SSS ^\pm _0$ are symplectic and $G$ is
preserved, they must be of the form~\eqref{def:scatteringcircular}.
Thus, it only remains to obtain formulas for the functions $f^\pm$.
To compute these functions, we use  \cite{DelshamsLS08}. Note that
this paper provides formulas for the scattering map in terms of the
Poincar\'e potential in a regular perturbation regime. The regular perturbative
approach used in~\cite{DelshamsKRS14} is to consider system~\eqref{eq:3bpInMcGehee} as an
order~$\mu$ perturbation of the two body problem. As has been explained in
Section \ref{sec:intro}, here we use a different nearly integrable regime,
also considered in~\cite{GuardiaMS14} for the
circular problem, which allow us to deal with  arbitrary values of~$\mu$. We
proceed as follows.

Recall that we already know that the scattering map is defined for
$\al\in\TT$ and $G\geq G_1$. To provide formulas for $f^\pm$ we need
to rescale the variables to transform the problem into a
perturbative setting. Consider the scaling
\[
 x = G_1^{-1} \wt x,\,\, y=G_1\ii \wt y,\,\,\al=\wt\al, \,\,G=G_1\wt G
\]
and the change of time $s = G_1^3 t$. It transforms
system~\eqref{eq:3bpInMcGehee} for $e_0=0$ into a system of the same
form with a new potential
\[
\wt \UU(\wt x, \wt \al,t; G_1) = G_1^2 \UU (G_1^{-1} \wt
x,\wt\al,G_1^3 t;0) = \frac{\wt x^2}{2} + \Delta \wt \UU(\wt x, \wt
\al, t; G_1)
\]
where $\Delta \wt \UU(\wt x, \wt \al, t; G_1)= G_1^2 \Delta \UU (G_1^{-1} \wt
x,\wt\al,t; G_1)$ and is given by
\[
\Delta \wt \UU(\wt x, \wt \al, t; G_1)=\frac{\wt x^2}{2}
\left(\frac{1-\mu}{\left(1-\frac{\mu }{G_1^2}\wt x^2
\cos \phi+(\frac{\mu }{2G_1^2}\wt x^2)^2\right)^{1/2}}
 +\frac{\mu}{\left(1+\frac{1-\mu }{G_1^2}\wt x^2\cos\phi+(\frac{1-\mu
}{2G_1^2}\wt x^2)^2\right)^{1/2}}
 -1\right)
\]
with $\phi =\wt \al-G_1^3 t$.
Expanding $\Delta \wt \UU$ in $G_1^{-1}$ one can see that it
can be written as $\Delta \wt \UU=G_1^{-4}\Delta \VV$ where
$\Delta \VV$ is bounded uniformly on~$G_1\to +\infty$. Therefore,
system~\eqref{eq:3bpInMcGehee} with $e_0=0$ in these rescaled variables can be
seen as a $\OO(G_1^{-4})$-perturbation of the two body problem.
Nevertheless, this limit is singular since $\Delta \wt \UU$, and
therefore $\Delta \VV$, is $2\pi/G_1^3$-periodic in time. When
$G_1\rightarrow+\infty$, the frequency of the perturbation blows
up. Hence, we need to adapt the theory developed in
\cite{DelshamsLS08} to our setting, along the lines
in~\cite{DelshamsLS00}

To transform this problem into a regular perturbation one, we introduce
$\de=G_1^{-4}$, which  we consider as an independent parameter.
Later we will recover the true value of $\delta$. We write the
potential $\wt \UU$ as
\[
\frac{\wt x^2}{2} +\de\Delta \VV
\]
and, then, system~\eqref{eq:3bpInMcGehee} is a
$O(\delta)$-perturbation, $2\pi/G_1^3$-periodic in time, of an
integrable system with a ``normally parabolic'' invariant manifold
$\wt \Lambda^{[G_1,+\infty)}$ with a $3$-dimensional homoclinic manifold that can be
parameterized by the homoclinic orbits
\[
\begin{aligned}
%
( x_\h(t;\wt G_0),  y_\h(t;\wt G_0), \al_\h(t;\wt \al_0, \wt
G_0), G_\h(t;\wt G_0),\wt s_0 +G_1^3 t),
\end{aligned}
\]
for any $\wt G_0$, $\wt \al _0$, $\wt s_0$. Here $(x_\h,y_\h,\al_\h,G_\h,s_\h)$ is the homoclinic orbit given
in Lemma~\ref{lemma:UnperturbedSeparatrix} and we have changed the
last component to take into account the period of the perturbation.

Since now system~\eqref{eq:3bpInMcGehee} is a $\delta$-regular perturbation
of an integrable system with a homoclinic manifold, we can use the
perturbative arguments (using deformation theory) in Theorem 32
of~\cite{DelshamsLS08}. We consider the Poincar\'e function
associated to the homoclinic manifold
\[
\begin{aligned}
 \wt L\left(\wt \al_0,\wt G_0,\wt s_0, \sigma\right) & =
\int_{-\infty}^{\infty} \Delta \VV\left(x_\h(t+\sigma;\wt G_0),\al_\h(t+\sigma;\wt\al_0,\wt G_0),\wt s_0 + G_1^3t\right)\, dt \\
& = G_1^4\int_{-\infty}^{\infty} \Delta \wt \UU\left(x_\h(t+\sigma;\wt G_0), \al_\h(t+\sigma;\wt \al_0,\wt G_0),\wt s_0 + G_1^3t\right)\, dt
\\
& = G_1^6\int_{-\infty}^{\infty}
\Delta \UU\left(G_1 ^{-1}x_\h(t+\sigma;G_1\wt G_0) , \al_\h(t+\sigma;\wt\al_0,\wt G_0),\wt  s_0 + G_1 ^3 t\right)\, dt \\
& = G_1^3\int_{-\infty}^{\infty}
\Delta \UU\left(G_1 ^{-1}x_\h\left( \frac{r+G_1 ^3\sigma}{G_1^3} ;\wt G_0\right) ,
\al_\h\left(\frac{r+G_1 ^3\sigma}{G_1^3};\wt\al_0,\wt G_0\right),\wt  s_0 + r\right)\, dr \\
& = G_1^3\int_{-\infty}^{\infty}
\Delta \UU\left(x_\h\left(r+G_1 ^3\sigma ;G_1 \wt G_0\right) , \al_\h\left(r+G_1
^3\sigma;\wt\al_0,G_1\wt G_0\right),\wt  s_0 + r\right)\, dr \\
& = G_1^3 L\left(\wt \al_0,G_1\wt G_0,\wt s_0,G_1^3\sigma;0\right)=
 G_1^3 L\left(\wt \al_0,G_1\wt G_0,\wt s_0-G_1^3\sigma,0;0\right),
\end{aligned}
\]
where we have used that the formulas in  Lemma \ref{lemma:UnperturbedSeparatrix}
imply
\[
G_1 ^{-1}x_\h\left( \frac{u}{G_1^3};\wt G_0\right)=x_\h\left(u;G_1\wt G_0\right), \
\al_\h\left(\frac{u}{G_1^3};\al_0,\wt G_0\right)= \al_\h\left(u;\al_0,G_1\wt G_0\right),
\]
and $L$ is the function  introduced in~\eqref{def:poincarepotential}. The
corresponding reduced Poincar\'e functions are obtained by evaluating
the function $\sigma\mapsto \wt L(\wt \al_0,\wt G,\wt s_0, \sigma)$ at its
nondegenerate critical points $\sigma^*_{\pm}$, given by $\wt s_0-\wt \al _0-G_1^3\sigma^* _-=0$ and $\wt s_0-\wt \al _0 +\pi-G_1^3\sigma^* _+=0$
(see~\eqref{def:sigmastarpm}). They are defined as
\[
\begin{aligned}
\wt \LL_-^*(\wt  \al_0,\wt G_0)&=\wt L(\wt \al_0,\wt
G_0,\wt s_0,G_1^{-3}(\wt s_0-\wt  \al_0))= G_1^3 L(\wt \al_0,G_1\wt G_0,\wt  \al_0,0;0)\\
 \LL_+^*(\wt  \al_0,\wt G_0)&=\wt L(\wt \al_0,\wt G_0,\wt s_0,G_1^{-3}(\pi+\wt s_0-\wt  \al_0))= G_1^3 L(\wt \al_0,G_1\wt G_0,\wt  \al_0-\pi,0;0).
\end{aligned}
\]
Since $L$, and, consequently, $\wt L$,  only depend on the angles
through $\al_0-s_0+\sigma$, $\wt \LL_\pm^*$ only depend on $ G_1 \wt G_0$.
\[
\begin{aligned}
\wt \LL_-^*( G_1\wt G_0)&=  G_1^3 L(0,G_1\wt G_0,0,0;0)\\
 \LL_+^*( G_1 \wt G_0)&= G_1^3 L(\pi,G_1\wt G_0,0,0;0).
\end{aligned}
\]

Then, the generating function of the scattering map for the rescaled system is given by
\begin{equation}\label{def:GeneratingFunction}
 S_\mathrm{resc}^\pm(\wt \al_0,\wt G_0)= \wt G_0\wt \al_0+\de\wt \LL_\pm^*( G_1 \wt G_0)+\OO(\de^2).
\end{equation}
%
Using the results in \cite{GuardiaMS14}, we have that
\[
 \LL_\pm^*( G_1 \wt G_0)=
G_1^3 L_0( G_1 \wt G_0)\pm G_1^3 L_{1,-1}( G_1 \wt G_0)+\OO\left(G_1^4(G_1\wt G_0)^{-\frac{3}{2}}e^{-2\frac{(G_1\wt
G_0)^3}{3}}\right)
\]
where
\[
\begin{split}
 L_0( G_1 \wt G_0)&=\mu(1-\mu)\frac{\pi}{2 ( G_1 \wt G_0)^3}\left(1+\OO((G_1\wt G_0)^{-4})\right)\\
L_{1,-1}( G_1 \wt G_0)&=
\mu(1-\mu)(1-2\mu)\sqrt{\frac{\pi}{8(G_1\wt
G_0)}}e^{-\frac{(G_1\wt G_0)^3}{3}}
\left(1+\OO((G_1\wt G_0)^{-1})\right).
\end{split}
\]
Using \eqref{def:GeneratingFunction}, we have that
the rescaled scattering map of the circular problem is of the form
\[
 \SSS_\mathrm{resc}^{\pm}\begin{pmatrix} \wt\al_0\\\wt G_0\end{pmatrix}=
 \begin{pmatrix}\wt\al_0+\wt f^\pm (\wt G_0)\\ \wt G_0\end{pmatrix}.
\]
with
\[
\wt f^\pm \left(\wt G_0\right)=\de\pa_{\wt G_0} \LL_\pm^*(G_1 \wt G_0)+\OO\left(\de^2\right)=
-\de\mu(1-\mu)\frac{3\pi}{2\wt G_0^4}+\OO\left(\de G_1^{-4}\wt
G^{-8}\right)+\OO\left(\de^2\right).
\]
where $\OO$ refers to uniform bounds $\wt G_0\in [1,+\infty)$.

We recover the value of $\de=G_1^{-4}$ and we scale back the
variables. Then, we obtain formula~\eqref{def:scatteringcircular}
with
\[
f^\pm (G)
= -\mu(1-\mu)\frac{3\pi}{2 G^4}+\OO\left( G_1^{-8}\right).
\]
To recover the claim of
Proposition~\ref{prop:ScatteringMapCircular}, we need to replace the
error term~$\OO\left( G_1^{-8}\right)$  by~$\OO\left(
G^{-8}\right)$. To do so, we consider the above procedure restricted
to compact sets. That is, for any $k\in \NN$, we choose $G\in [k
G^*, (k+1)G^*]$ and thus we take $\delta_k = G_1^{-4} = (k G^*
)^{-4}$. Then, for any $G\in [k G^*, (k+1)G^*]$, $\delta_k < C
G^{-4}$, for some $C>0$ independent of~$k$.

\bibliography{references}
\bibliographystyle{alpha}
\end{document}